\documentclass[10pt]{amsart}

\usepackage[colorlinks]{hyperref}
\usepackage{color,graphicx,shortvrb}
\usepackage[latin 1]{inputenc}

\usepackage[active]{srcltx} 

\usepackage{enumerate}
\usepackage{amssymb}

\newtheorem{theorem}{Theorem}[section]

\newtheorem{corollary}[theorem]{Corollary}

\newtheorem{lemma}[theorem]{Lemma}

\newtheorem{proposition}[theorem]{Proposition}
\newtheorem{remark}[theorem]{Remark}

\def\J#1#2#3{ \left\{ #1,#2,#3 \right\} }

\def\NN{{\mathbb{N}}}
\def\11{\textbf{$1$}}

\DeclareGraphicsExtensions{.jpg,.pdf,.png,.eps}

\begin{document}

\title{Low rank compact operators and Tingley's problem}

\author[F.J. Fern\'{a}ndez-Polo]{Francisco J. Fern\'{a}ndez-Polo}
\author[A.M. Peralta]{Antonio M. Peralta}

\address{Departamento de An{\'a}lisis Matem{\'a}tico, Facultad de
Ciencias, Universidad de Granada, 18071 Granada, Spain.}
\email{pacopolo@ugr.es}
\email{aperalta@ugr.es}




\date{}

\begin{abstract} Let $E$ and $B$ be arbitrary weakly compact JB$^*$-triples whose unit spheres are denoted by $S(E)$ and $S(B)$, respectively. We prove that every surjective isometry $f: S(E) \to S(B)$ admits an extension to a surjective real linear isometry $T: E\to B$. This is a complete solution to Tingley's problem in the setting of weakly compact JB$^*$-triples. Among the consequences, we show that if $K(H,K)$ denotes the space of compact operators between arbitrary complex Hilbert spaces $H$ and $K$, then every surjective isometry $f: S(K(H,K)) \to S(K(H,K))$ admits an extension to a surjective real linear isometry $T: K(H,K)\to K(H,K)$.
\end{abstract}

\maketitle
\thispagestyle{empty}

\section{Introduction}

It does not seem an easy task to write an introductory paragraph for a problem which has been open since 1987. As in most important problems, the precise question is easy to pose and reads as follows:  Let $X$ and $Y$ be normed spaces, whose unit spheres are denoted by $S(X)$ and $S(Y)$, respectively. Suppose $f:S(X)\to S(Y)$ is a surjective isometry.  The so-called \emph{Tingley's problem} asks wether $f$ can be extended to a real linear (bijective) isometry $T : X \to Y$ between the corresponding spaces (see \cite{Ting1987}). \smallskip

The problem was named after D. Tingley proved in \cite[THEOREM, page 377]{Ting1987} that every surjective isometry $f: X\to Y$ between the unit spheres of two finite dimensional spaces satisfies $f(-x) = -f(x)$ for every $x\in S(X)$. \smallskip

Readers interested in a classic motivation, can sail back to the celebrated Mazur-Ulam theorem asserting that every surjective isometry between two normed spaces over $\mathbb{R}$ is a real affine function. In a subsequent paper,  P. Mankiewicz established in \cite{Mank1972} that, given two convex bodies $V\subset X$ and $W\subset Y$, every surjective isometry $g$ from
$V$ onto $W$ can be uniquely extended to an affine isometry from $X$ onto $Y$. Consequently, every surjective isometry between the closed unit balls of two Banach spaces $X$ and $Y$ extends uniquely to a real linear isometric isomorphism from $X$ into $Y$.\smallskip

With this historical background in mind, Tingley's problem asks whether the conclusion in Mankiewicz's  theorem remains true when we deal with the unit sphere whose interior is empty. In other words, when an isometric identification of the unit spheres of two normed spaces can produce an isometric (linear) identification of the spaces.\smallskip

As long as we know, Tingley's problem remains open even for surjective isometries between the unit spheres of a pair of 2 dimensional Banach spaces. However, positive answers haven been established in a wide range of classical Banach spaces. In an interesting series of papers, G.G. Ding proved that Tingley's problem admits a positive answer for every surjective isometry  $f: S(\ell^p (\Gamma)) \to S(\ell^p (\Delta))$ with $1\leq p \leq \infty$ (see \cite{Ding2002,Di:p,Di:8} and \cite{Di:1}). More recently, D. Tan showed that the same conclusion remains true for every surjective isometry $f : S(L^{p}(\Omega, \Sigma, \mu)) \to S(Y)$, where $(\Omega, \Sigma, \mu)$ is a $\sigma$-finite measure space, $1\leq p\leq \infty$, and $Y$ is a Banach space (compare \cite{Ta:8, Ta:1} and \cite{Ta:p}). A result of R.S. Wang in \cite{Wang} proves  that for each pair of locally compact Hausdorff spaces $L_1$ and $L_2$, every surjective isometry $f: S(C_0(L_1)) \to S(C_0(L_1))$ admits an extension to a surjective real linear isometry from $C_0(L_1)$ onto $C_0(L_1)$. V. Kadets and M. Martín gave another positive answer to Tingley's problem in the case of finite dimensional polyhedral Banach spaces (see \cite{KadMar2012}). The surveys \cite{Ding2009} and \cite{YangZhao2014} contain a detailed revision of these results and additional references.\smallskip

In the setting of C$^*$-algebras, R. Tanaka recently establishes in \cite{Tan2016-2} that every surjective isometry from the unit sphere of a finite dimensional C$^*$-algebra $N$ into the unit sphere of another C$^*$-algebra $M$ admits a unique extension to a surjective real linear isometry from $N$ onto $N$. More recently, Tanaka also proves in \cite{Tan2016preprint} that the same conclusion holds when $N$ and $M$ are finite von Neumann algebras.\smallskip

As a result of a recent collaboration between the second author of this note and R. Tanaka (see \cite{PeTan16}), new positive answers to Tingley´s problem have been revealed for spaces of compact operators. Concretely, denoting by $K(H)$ the C$^*$-algebra of all compact operators on a complex Hilbert space $H$, it is shown that for every pair of complex Hilbert spaces $H$ and $H'$, every surjective isometry $f: S(K(H))\to S(K(H'))$ admits a unique extension to a real linear isometry $T$ from $K(H)$ onto $K(H')$; and the same conclusion also holds when $K(H)$ and $K(H')$ are replaced by arbitrary compact C$^*$-algebras (compare \cite[Theorem 3.14]{PeTan16}). The novelties in the just quoted note rely on the introduction of Jordan techniques to tackle Tingley's problem. From the wider point of view of weakly compact JB$^*$-triples, it is established that if $E$ and $B$ are weakly compact JB$^*$-triples not containing direct summands of rank smaller than or equal to 3, and with rank greater than or equal to 5, then surjective isometry $f: S(E) \to S(B)$ admits a unique extension to a surjective real linear isometry from $E$ onto $B$ (compare \cite[Theorem 3.13]{PeTan16} and see below for the concrete definitions of weakly compact JB$^*$-triples).\smallskip

The case of surjective isometries between the unit spheres of two weakly compact JB$^*$-triples of rank between 2 and 4 was a left as an open problem in \cite{PeTan16}. This problem affects particularly interesting cases including spin factors, finite dimensional Cartan factors, and space of the form $K(H,H')$ where  $H$ and $H'$ are complex Hilbert spaces with dim$(H)=\infty$ and $2\leq $dim$(H')\leq 4$. In this paper we complete the study left open in \cite{PeTan16}, solving Tingley's problem in the remaining cases of weakly compact JB$^*$-triples of low rank and providing a complete solution of Tingley's problem for surjective isometries between the unit spheres of two weakly compact JB$^*$-triples. Besides Jordan techniques, the arguments applied in this note are strongly based on new geometric properties for Cartan factors, and results in Functional Analysis and Operator algebras. \smallskip

In order to have a precise idea of the results explored in this paper, it seems necessary to arrange the definition of those complex Banach spaces called \emph{Cartan factors}.  All Banach spaces considered in this note are complex. Let $H$ and $H^\prime$ be complex Hilbert spaces, and let $j : H\to H$ be a conjugation (i.e., a conjugate linear isometry of order two) on $H$. A Banach space is called a \emph{Cartan factor of type 1} if it coincides with the space $L(H,H^\prime)$ of all bounded linear operators from $H$ into $H^\prime$. To understand the next Cartan factors let $^t: L(H)\to L(H)$ be the complex linear involution on $L(H)$ defined by $x^t = j x^* j$ ($x\in L(H)$). The spaces $$C_2 =\{x\in L(H) : x^t =-x\} \hbox{ and } C_3 =\{x\in L(H) : x^t =x\}$$ are closed subspaces of $L(H)$ called \emph{Cartan factors of type 2 and 3}, respectively. A Banach space $X$ is called a \emph{Cartan factor of type 4} or \emph{spin} if $X$ admits a complete inner product $(.|.)$ and a conjugation $x\mapsto \overline{x},$ for which the norm of $X$ is given by $$ \|x\|^2 = (x|x) + \sqrt{(x|x)^2 -|
(x|\overline{x}) |^2}.$$ \emph{Cartan factors of types 5 and 6} (also called \emph{exceptional} Cartan factors) are all finite dimensional. We refer to \cite{Chu2012} for additional details.\smallskip

According to the terminology employed by L. Bunce and  Ch.-H. Chu in \cite{BuChu}, associated with each Cartan factor we find an \emph{elementary JB$^*$-triple}. The elementary JB$^*$-triples of types 1 to 6 ($K_1, \ldots, K_6$) are defined as follows: types 4, 5 and 6 are precisely the Cartan factors of the same type, i.e. $K_4 = C_4,$ $K_5 = C_5$, and $K_6 = C_6$. For the other types, we introduce some additional notation. The symbol $K(H,H^\prime)$ will denote the compact linear operators from $H$ into $H^\prime$. The elementary JB$^*$-triples of types 1, 2 and 3 are $K_1 =K(H,H^\prime)$, $K_2=C_2\cap K(H),$ and $K_3=C_3\cap K(H)$, respectively. \smallskip

Elementary JB$^*$-triples are the building blocks of weakly compact JB$^*$-triple, more precisely, every weakly compact JB$^*$-triple can be written as a $c_0$-sum of a family of elementary
JB$^*$-triples (see \cite[Lemma 3.3 and Theorem 3.4]{BuChu}).\smallskip

The main result in this note (Theorem \ref{thm Tingley thm ofr weakly compact JB*-triples}) establishes that every surjective isometry $f: S(E) \to S(B)$ between the unit spheres of two arbitrary weakly compact JB$^*$-triples admits a unique extension to a surjective real linear isometry $T: E\to B$. The proof, which requires a substantial technical novelty, is divided into different partial results. In Section \ref{sec: sintesis} we explore Tingley's problem when the domain is an elementary JB$^*$-triple of type 1, 2, 3, spin or finite dimensional (see Theorems \ref{thm Tingley thm for type 1 rank 2, 3 and 4}, \ref{thm Tingley thm for type 2 rank 2, 3 and 4}, \ref{thm Tingley thm for type 3 rank 2, 3 and 4}, \ref{thm Tingley thm for finite dimensional}, and \ref{thm Tingley thm for spin factors}). Previously, in Section \ref{sec: geometric properties}, we establish additional properties satisfied by every surjective isometry between the unit spheres of two elementary JB$^*$-triples. This geometric study complete the result obtained in \cite{PeTan16} for elementary JB$^*$-triples with rank greater than or equal to 5. For completeness reasons, the structure of the paper is supplemented with a subsection with basic notions, results and references applied in this note.\smallskip

It seems natural to conclude this introduction with a desiderata; once Tingley's problem has been solved for compact operators, compact C$^*$-algebras and weakly compact JB$^*$'triples, the unexplored frontier appears now in the case of $B(H)$, general von Neumann algebras and C$^*$-algebras, and other operator spaces.

\subsection{Basic background on JB$^*$-triples}

In this subsection we revisit the basic notions and result in JB$^*$-triples. The reader who is familiar with these concepts can simply skip this part.\smallskip

Throughout this note the symbol $\mathcal{B}_X$ will stand for the closed unit ball of a Banach space $X$.\smallskip

Let $H$ and $H^\prime$ a pair of complex Hilbert spaces. In most of cases, the space $L(H,H^\prime)$ is not a C$^*$-algebra. An example of this claim appears when $H$ is finite dimensional and $H^\prime$ is infinite dimensional. However, the operator norm and the triple product defined by \begin{equation}\label{eq triple product for operators} \{a,b,c\}=\frac12( ab^* c + c b^* a) \ (a,b,c\in L(H,H^\prime))
 \end{equation} equip $L(H,H^\prime)$ with a structure of JB$^*$-triple in the sense introduced by W. Kaup in \cite{Ka83}. A \emph{JB$^*$-triple} is a complex Banach space $E$ admitting a continuous triple product $\{a,b,c\}$ which is conjugate linear in $b$ and linear and symmetric in $a$ and $c$, and satisfies the following axioms:\begin{enumerate}[(JB$^*$1)]\item $L(a,b) L(c,d) - L(c,d) L(a,b) = L(L(a,b) (c),d) - L(c,L(b,a)(d))$, for every $a,b,c,d$ in $E$, where $L(a,b)$ is the operator on $E$ defined by $L(a,b) (x) =\{a,b,x\}$;
\item $L(a,a)$ is an hermitian operator on $E$ with non-negative spectrum;
\item $\|\{a,a,a\}\| = \|a \|^3$, for every $a\in E$.
\end{enumerate}

Examples of JB$^*$-triples include the spaces $L(H,H^\prime)$, the spaces $K(H,H^\prime)$ of all compact operators between two complex Hilbert spaces, complex Hilbert spaces, and all C$^*$-algebras equipped with the same product defined in \eqref{eq triple product for operators}. JB$^*$-triples constitute a category which produces a Jordan model valid to generalize C$^*$-algebras. Every JB$^*$-algebra is a JB$^*$-triple under the triple product $$\{a,b,c\}= (a \circ b^*)\circ c + (c\circ b^*)\circ a - (a\circ c) \circ b^*.$$

Cartan factors and elementary JB$^*$-triples of types 1, 2 and 3 are JB$^*$-triples with respect to the product given in \eqref{eq triple product for operators}. Spin factors are JB$^*$-triples when equipped with the product $\{x, y, z\} = (x|y)z + (z|y) x -(x|\overline{z})\overline{y}.$ The exceptional Cartan factor $C_6$ is a JB$^*$-algebra and $C_5$ is a JB$^*$-subtriple of $C_6$.\smallskip

A closed subtriple $I$ of a JB$^*$-triple $E$ is called an \emph{ideal} in $E$ if $\{I,E,E\} + \{E,I,E\} \subseteq I$ holds. A JB$^*$-triple which cannot be decomposed into the direct sum of two non-trivial closed ideals is called a \emph{JB$^*$-triple factor}. \smallskip

Those JB$^*$-triples which are also dual Banach spaces are called \emph{JBW$^*$-triples}. Therefore, von Neumann algebras are examples of JBW$^*$-triples. Analogously as in the case of von Neumann algebras, JBW$^*$-triples admit a unique (isometric) predual and their triple product is separately weak$^*$-continuous (see \cite{BarTi86}). The bidual, $E^{**}$, of a JB$^*$-triple, $E$, is a JBW$^*$-triple with a triple product extending the triple product of $E$ (cf. \cite{Di86}). Cartan factors are all JBW$^*$-triple factors.\smallskip

Partial isometries in a C$^*$-algebra can be generalized to tripotents in a JB$^*$-triple. More concretely, an element $u$ in a JB$^*$-triple $E$ is a \emph{tripotent} if $\{u,u,u\}=u$. The set of all tripotents in a JB$^*$-triple $E$ will be denoted by the symbol $\mathcal{U} (E)$. For each tripotent $u$ in $E$, the mappings $P_{i} (u) : E\to E_{i} (u)$, $(i=0,1,2)$, defined by
\begin{equation}\label{eq pepe} P_2 (u) = L(u,u)(2 L(u,u) -id_{E}), \ P_1 (u) = 4
L(u,u)(id_{E}-L(u,u)),
\end{equation}
$$\ \hbox{ and } P_0 (u) =
(id_{E}-L(u,u)) (id_{E}-2 L(u,u)),$$
 are contractive linear projections. Each $P_j (u)$ is known as the \emph{Peirce-$j$} projection induced by $u$, and corresponds to the projection of $E$ onto the eigenspace $E_j(u)$ of $L(u, u)$ corresponding to the eigenvalue $\frac{j}{2}$. The Peirce decomposition of $E$ relative to $u$ writes $E$ in the form $$E= E_{2} (u) \oplus E_{1} (u)
\oplus E_0 (u).$$ The following \emph{Peirce rules} are satisfied, \begin{equation} \label{eq peirce rules1} \J {E_{2}
(u)}{E_{0}(u)}{E} = \J {E_{0} (u)}{E_{2}(u)}{E} =0,
\end{equation} \begin{equation}
\label{eq peirce rules2} \J {E_{i}(u)}{E_{j} (u)}{E_{k}
(u)}\subseteq E_{i-j+k} (u),
\end{equation} where $E_{i-j+k} (u)=0$ whenever
$i-j+k \notin \{ 0,1,2\}$ (compare \cite{FriRu85}). It follows from \eqref{eq pepe} and the separate weak$^*$-continuity of the triple product that Peirce projections associated with a tripotent in a JBW$^*$-triple are weak$^*$ continuous. It is further known that $E_2(u)$ is a JB$^*$-algebra with product and involution determined by $a\circ_{u} b =\{a,u,b\}$ and $a^{\sharp_{u}} =\{u,a,u\}$, respectively. \smallskip

A non-zero tripotent $u$ is called \emph{minimal} if $E_2 (u) = \mathbb{C} u$, \emph{complete} if the Peirce subspace $E_0(u)$ vanishes, and \emph{unitary} if $E = E_2 (e)$.\smallskip

Two tripotents $u,e$ in $E$ are said to be \emph{orthogonal} (written $e\perp u$) if $\{e,e,u\} =0$ (or equivalently, $\{u,u,e\} =0$). Actually, the relation $\perp$ can be considered on the whole $E$ by defining $a\perp b$ if $\{a,a,b\}=0$ (see \cite[\S 1]{BurFerGarMarPe} for additional details). It is known that orthogonal elements in $E$ are geometrically $M$-orthogonal (see \cite[Lemma 1.3]{FriRu85} and \cite[Lemma 1.1]{BurFerGarMarPe}), i.e. \begin{equation}\label{eq orthgonal implies M-orthogonal} a\perp b \Rightarrow \| a\pm b \| = \max \{ \| a\|, \|b\|\}.
\end{equation} The relation of orthogonality is adequate to define a partial order in $\mathcal{U} (E)$ given by $u\leq e$ if $e-u$ is a tripotent orthogonal to $e$ (see, for example, \cite{FriRu85} or \cite{Horn87}). The set $\mathcal{U} (E)$ is a lattice with respect to the partial order $\leq$. A tripotent $u$ in a JBW$^*$-triple $W$ is minimal if and only if it is a minimal element in the lattice $(\mathcal{U} (E),\leq)$.\smallskip

A subset $\mathcal S$ in a JB$^*$-triple $E$ is called \emph{orthogonal} if $0\notin \mathcal{S}$ and $x\perp y$ for every $x\neq y$ in $\mathcal{S}$. The \emph{rank} of $E$ (denoted by rank$(E)$ or by $r(E)$) is the minimal cardinal $r$ satisfying that $\sharp \mathcal{S}\leq r$ for every orthogonal subset of $E$ (compare \cite{Ka97}). The rank of a tripotent $e$ in $E$ is defined as the rank of the JB$^*$-triple $E_2(e)$. JB$^*$-triples of finite rank are all reflexive, they are described, for example in \cite{BeLoPeRo}. It is known that in a Cartan factor $C$, the rank of $C$ is precisely the minimal cardinal $r$ satisfying that $\sharp \mathcal{S}\leq r$ for every orthogonal subset of minimal tripotents in $E$ (see \cite[\S 3 or Theorem 5.8]{Ka97} and \cite{BeLoPeRo}).\smallskip

Let $u,v$ be tripotents in a JB$^*$-triple $E$. We say that $u$ and $v$ are \emph{collinear} (written $u\top v$) if $u\in E_1(v)$ and $v\in E_1(u)$.\smallskip

Given $u,v\in \mathcal{U} (E)$ with $u \in E_l(v)$ for some $l=0,1,2$ then \begin{equation}\label{eq compatible tripotents} P_k(u) P_j (v) = P_j (v) P_k (u),
 \end{equation} for all $j,k\in \{0,1,2\}$ (compare \cite[$(1.9)$ and $(1.10)$]{Horn87}).\smallskip

There are JB$^*$-triples $E$ for which $\mathcal{U} (E)$ is empty. However, since in a JB$^*$-triple $E$ the extreme points of its closed unit ball are precisely the complete tripotents in $E$ (see \cite[Proposition 3.3]{KaUp}), we can obviously conclude, via Krein-Milman theorem, that every JBW$^*$-triple contains a huge set of tripotents.\smallskip

For each normal functional $\varphi$ in the predual of a JBW$^*$-triple $W$, there exists a unique tripotent $u\in W$ (called the \emph{support tripotent} of $\varphi$ in $W$) satisfying $\varphi = \varphi P_2 (u)$ and $\varphi|_{M_2 (u)}$ is a faithful positive normal functional on the JBW$^*$-algebra $W_2 (u)$ (see \cite[Proposition 2]{FriRu85}). The support tripotent of a normal functional $\varphi$ will be denoted by $e(\varphi)$. Suppose $u$ is another tripotent in $W$ such that $\varphi (u) =1 =\|\varphi\|$. It follows from the arguments in \cite[Propositions 1 and 2]{FriRu85} and their proofs that $u = s(\varphi) +P_0 (s(\varphi)) (u)$ (i.e. $e\leq u$). Actually, the same arguments show the following:
\begin{equation}\label{eq supportd at the support} x\in W, \hbox{ with } \|x\| =1= \varphi (x) =\|\varphi\| \Rightarrow x = s(\varphi) +P_0 (s(\varphi)) (x).
\end{equation}

By Proposition 4 in \cite{FriRu85} we also know that minimal tripotents in a JBW$^*$-triple $W$ are precisely the support tripotents of the extreme points of the closed unit ball of its predual.\smallskip

Suppose $x$ is an element in a JB$^*$-triple $E$. The symbol $E_x$ will denote the JB$^*$-subtriple generated by $x$, that is, the closed subspace generated by all odd powers of the form $x^{[1]} := x$, $x^{[3]} := \J xxx$, and $x^{[2n+1]} := \J xx{x^{[2n-1]}},$ $(n\in \NN)$. It is known that
$E_x$ is JB$^*$-triple isomorphic (and hence isometric) to a commutative C$^*$-algebra in which $x$ is a positive generator (cf. \cite[Corollary 1.15]{Ka83}). Suppose $x$ is a norm-one element. The sequence $(x^{[2n -1]})$ converges in the weak$^*$-topology of $E^{**}$ to a tripotent (called the \hyphenation{support}\emph{support} \emph{tripotent} of $x$) $u(x)$ in $E^{**}$ (see \cite[Lemma 3.3]{EdRutt88} or \cite[page 130]{EdFerHosPe2010}).\smallskip

C.M. Edwards and G.T. R\"{u}ttimann developed in \cite{EdRu96} an analogue to the notion of compact projections in the bidual of a C$^*$-algebra in the more general setting of JB$^*$-triples. A tripotent $e$ in the second dual, $E^{**}$, of a JB$^*$-triple $E$ is said to be \emph{compact-$G_{\delta}$} if there exists a norm one element $a$ in $E$ satisfying $u(a) =e$. A tripotent $e$ in $E^{**}$ is called \emph{compact} if $e=0$ or it is the infimum of a decreasing net of compact-$G_{\delta}$ tripotents in $E^{**}$. The symbol $\mathcal{U}_{c} (E^{**})$ will stand for  the lattice consisting of those compact tripotents in $E^{**}$ equipped with the order $\leq$.\smallskip

Theorem 3.4 in \cite{BuFerMarPe} shows that minimal tripotents in $E^{**}$ are compact, and consequently the relation $\hbox{min } \mathcal{U} (E^{**}) \subseteq \hbox{min } \mathcal{U}_{c} (E^{**})$ holds, where ``min'' denotes the minimal elements in the corresponding lattice. It is also observed in \cite[page 47, comments after Theorem 3.4]{BuFerMarPe} that \begin{equation}\label{eq minimal and mimimal compact coincide} \hbox{min } \mathcal{U} (E^{**}) = \hbox{min } \mathcal{U}_{c} (E^{**}).
\end{equation}

The weak$^*$-closed ideal $J(v)$ generated by a minimal tripotent $v$ in a JBW$^*$-triple $W$, is precisely a Cartan factor. It is further know that $W$ writes as the direct orthogonal sum of $J(v)$ and another weak$^*$-closed of $W$ (see \cite[MAIN THEOREM in page 302]{DanFri87}).

\section{ Tingley's problem for weakly compact JB$^*$-triples which are not factors}\label{sec:2}

In order to resume the study of surjective isometries between the unit spheres of two weakly compact JB$^*$-triples, we start by recalling a result in \cite{PeTan16}.

\begin{proposition}\label{p surjective isometries between the spheres preserve finite rank tripotents}\cite[Propositions 3.2 and 3.9]{PeTan16} Let $E$ and $B$ be weakly compact JB$^*$-triples, and  suppose that $f: S(E) \to S(B)$ is a surjective isometry. Then the following statements hold:\begin{enumerate}[$(a)$] \item For each finite rank tripotent $e$ in $E$ there exists a unique finite rank tripotent $u$ in $B$ such that $f( (e + \mathcal{B}_{ E_0^{**}(e)}) \cap E )= (u + \mathcal{B}_{B_0^{**}(u)}) \cap B$;
\item The restriction of $f$ to each norm closed face of $\mathcal{B}_{E}$ is an affine function;
\item For each finite rank tripotent $e$ in $E$ there exists a unique finite rank tripotent $u$ in $B$ and a surjective real linear isometry $T_e : E_0(e) \to B_0(u)$ such that $$f(e +x )= u + T_e (x),$$ for every $x\in \mathcal{B}_{E_0(e)}$, and $T_e (x) = f(x)$ for all $x\in S(E_0(e))$;
\item For each finite rank tripotent $e$ in $E$ there exists a unique finite rank tripotent $u$ in $B$ such that $f(e )= u$.$\hfill\Box$
\end{enumerate}
\end{proposition}

One of the obstacles to the questions left open in \cite{PeTan16} is mainly due to the fact that the arguments in \cite[Lemma 3.4]{PeTan16} are not valid for arbitrary weakly compact JB$^*$-triples. We begin the new contributions with a generalization of \cite[Lemma 3.4]{PeTan16} for general JB$^*$-triples.

\begin{proposition}\label{p l 3.4 new} Let $e$ and $x$ be norm-one elements in a JB$^*$-triple $E$. Suppose that $e$ is a minimal tripotent and $\|e-x\| = 2$. Then $x= -e + P_0(e) (x)$.
\end{proposition}

\begin{proof} The element $\frac{e-x}{2}$ has norm-one in $E$.  Let $u= s(\frac{e-x}{2})$ denote the support tripotent of $\frac{e-x}{2}$ in the JBW$^*$-triple $E^{**}$. It is known that $u$ is a compact(-$G_\delta$) tripotent in $E^{**}$ and $\frac{e-x}{2} = u + P_0 (u) (\frac{e-x}{2})$ (see \cite[\S 4]{EdRu96}). Since the lattice $\mathcal{U}_{c} (E^{**})$ of those compact tripotents in $E^{**}$ is atomic (see \cite[Theorem 4.5$(ii)$]{EdRu96}), we can find at least a minimal element $v\in \hbox{min }\mathcal{U}_{c} (E^{**})=\hbox{min } \mathcal{U} (E^{**})$ satisfying $v\leq u$. Therefore, we can write $\displaystyle \frac{e-x}{2} = v + P_0 (v) (\frac{e-x}{2})$ in $E^{**}$, and hence $$\frac12 P_2 (v) (e) + \frac12 P_2 (v) (-x) =  P_2 (v) (\frac{e-x}{2}) = v.$$ Since $v$ is the unit of $E^{**}_2 (v) = \mathbb{C} v$, and hence it is an extreme point of the closed unit of the latter space, we deduce that $P_2 (v) (e) = P_2 (v) (-x) = v$. The minimality of $e$ in $E$ (and in $E^{**}$) shows that $e=v$ (compare \cite[Corollary 1.7]{FriRu85}). Having in mind that $x$ is a norm-one element in $E$, it follows from \cite[Lemma 1.6]{FriRu85} that $ x = -v + P_0 (v) (x) = -e + P_0 (e) (x) $ in $E$.
\end{proof}

Many important consequences can be now derived from the above strengthened version of \cite[Lemma 3.4]{PeTan16}. For example, when in the proof of \cite[Theorem 3.6]{PeTan16} we replace \cite[Lemma 3.4]{PeTan16} with Proposition \ref{p l 3.4 new}, the arguments remain valid, word-by-word, to prove the following:

\begin{theorem}\label{t Tingley antipodes for finite rank new} Let $f: S(E) \to S(B)$ be a surjective isometry between the unit spheres of two weakly compact JB$^*$-triples. Suppose $e$ is a finite rank tripotent in $E$. Then $f(-e)=-f(e)$. Furthermore, if $e_1,\ldots,e_m$ are mutually orthogonal minimal tripotents in $E$, then $f(e_1),\ldots,f(e_m)$ are mutually orthogonal minimal tripotents and $$ f(e_1+\ldots+e_m) = f(e_1)+\ldots+f(e_m).$$ Consequently, $f$ maps tripotents of rank $k$ to tripotents of rank $k$, $f$ is additive on mutually orthogonal finite rank tripotents, and the ranks of $E$ and $B$ coincide. $\hfill\Box$
\end{theorem}

We can also remove now part of the hypothesis in \cite[Theorem 3.12]{PeTan16}.

\begin{theorem}\label{thm Tyngley co sums with more than one element new} Let $f: S(E) \to S(B)$ be a surjective isometry between the unit spheres of two weakly compact JB$^*$-triples. Suppose that $\displaystyle E=\oplus^{c_0}_{j\in J} K_j$, where $\sharp J \geq 2$, and every $K_j$ is an elementary JB$^*$-triple. Then there exists a surjective real linear isometry $T: E\to B$ satisfying $T|_{S(E)} = f$.
\end{theorem}

\begin{proof} If we replace \cite[Theorem 3.6]{PeTan16} with Theorem \ref{t Tingley antipodes for finite rank new} above, the arguments in the proof of \cite[Theorem 3.12]{PeTan16} also hold in this setting.
\end{proof}

Let $f: S(E) \to S(B)$ be a surjective isometry between the unit spheres of two weakly compact JB$^*$-triples. Suppose that $\displaystyle E=\oplus^{c_0}_{j\in \mathcal{J}} K_j$ and $\displaystyle B=\oplus^{c_0}_{k\in \mathcal{K}} K^{\prime}_k$, where the $K_j$ and the $K^{\prime}_k$ are elementary JB$^*$-triples. If $\sharp \mathcal{J} \geq 2$ or $\sharp \mathcal{K} \geq 2$ we deduce from Theorem \ref{thm Tyngley co sums with more than one element new} above that $f$ can be extended to a real linear (bijective) isometry $T : E \to F$. We can therefore restrict our attention to the case of a surjective isometry between the unit spheres of two elementary JB$^*$-triples.\smallskip

Let $f: S(K) \to S(K^{\prime})$ be a surjective isometry between the unit spheres of two  elementary JB$^*$-triples. Theorem \ref{t Tingley antipodes for finite rank new} implies that $K$ and $K^{\prime}$ both have the same rank. If rank$(K)= 1$  or rank$(K)\geq 5$, Corollary 3.15 and Theorem 3.13 in \cite{PeTan16} prove the existence of a surjective real linear isometry $T : K\to K^{\prime}$ which coincides with $f$ on the unit sphere of $K$. We can actually restrict our attention to the cases in which $2\leq$rank$(K)\leq 4$. Theorem \ref{thm Tingley thm ofr weakly compact JB*-triples} below assures (even under the weaker hypothesis $2\leq$rank$(K)$) the existence of a surjective complex linear or conjugate linear isometry $T: K\to K'$ satisfying $T(x) = f(x)$ for every $x\in S(K)$. These arguments offer a complete answer to Tingley's problem for weakly compact JB$^*$-triples.

\begin{theorem}\label{thm Tyngley general weakly compact JB*-triples} Let $f: S(E) \to S(B)$ be a surjective isometry between the unit spheres of two weakly compact JB$^*$-triples. Then there exists a surjective real linear isometry $T: E\to B$ satisfying $T|_{S(E)} = f$.$\hfill\Box$
\end{theorem}

\section{Geometric properties of the subtriple generated by two minimal tripotents}\label{sec: geometric properties}

Our first result can be easily derived from the \emph{Triple System Analyzer} \cite[Proposition 2.1]{DanFri87}.

\begin{lemma}\label{l collinearity}\cite[Proposition 2.1 $(i)$]{DanFri87} Let $e_1$ and $e_2$ be tripotents in a JB$^*$-triple $E$. The following statements hold:\begin{enumerate}[$(a)$] \item If $e_1 \top e_2$ then $e_1$ is minimal if and only if $e_2$ is;
\item If $e_1$ and $e_2$ are minimal and $e_2\in E_1 (e_1)$ then $e_1 \top e_2$. $\hfill\Box$
\end{enumerate}
\end{lemma}

Our next result is also essentially contained in \cite{DanFri87}. We include an explicit statement with a justification for completeness.

\begin{lemma}\label{l technical headache} Let $C$ be a Cartan factor of rank greater or equal than 2. Let $e_1$ and $e_2$ be minimal tripotent in $C$ with $e_1\top e_2$. Then there exists a minimal tripotent $u$ in $C$ such that $e_1\perp u$, and $u \top e_2$.
\end{lemma}

\begin{proof} By \cite[Corollary 2.2]{DanFri87} the JBW$^*$-triple $C_1(e_2)$ has rank one or two.\smallskip

Suppose first that $C_1(e_2)$ has rank one. Since $C_1(e_2)\ni e_1\top e_2$, we deduce from \cite[PROPOSITION in page 305]{DanFri87} that $C$ is isometric to a Hilbert space, and hence $C$ must have rank one, which is impossible. Therefore $C_1(e_2)$ has rank two, and we can thus find a tripotent $u$ in $C_1(e_2)$ such that $u\perp e_1$ and $u$ is minimal in $C_1(e_2)$. \smallskip

Clearly, $e_1+u$ is a tripotent in $C_1(e_2)$ which is not minimal in the latter JB$^*$-triple. By the \emph{Triple System Analyzer} \cite[Proposition 2.1 $(iii)$]{DanFri87}, there exist two minimal tripotents $u_1$, $u_2$ in $C$ satisfying $e_1+u = u_1+u_2$, in other words, $e_1+u$ is a rank two tripotent in $C$, and hence $u$ must be minimal in $C$ too.
\end{proof}

For later purposes, we shall make use of another decomposition which is also associated with a tripotent $u$ in a JB$^*$-triple $E$.  Since the mapping $Q(u)|_{E_2(u)}$ is the involution of the JB$^*$-algebra $E_2(u)$, it follows that $E_{2} (u) = E^{1} (u) \oplus E^{-1} (u)$, where $E^{k} (u) := \{ x\in E : Q(u) (x) := \{u,x,u\}= k x \}$. Clearly, $E^{-1} (u) = i E^{1} (u)$. By Peirce rules, $E_{0} (u) \oplus E_{1} (u)$ coincides with the kernel $\ker(Q(u)) = E^{0} (u)$. The identity $$ \{ {E^{i} (u)},{E^{j} (u)},{E^{k} (u)}\} \subseteq E^{i j k} (u),$$ holds whenever  $i j k \ne 0$. The spaces $E^{k}(u)$ induce the following decomposition $$E= E^{-1} (u) \oplus E^{0} (u) \oplus E^{1} (u).$$

Let $u$ and $v$ be a couple of arbitrary minimal tripotents in a JB$^*$-triple $E$. The JB$^*$-subtriple $J$ of $E$ generated by $u$ and $v$ was totally described by Y. Friedman and B. Russo \cite[Lemma 2.3 and Proposition 5]{FriRu85}. It follows from the just quoted results that $J$ is linearly spanned by $u$, $v$, $P_1 (u) (v)$, and $P_1 (v) (u)$. It is further known that $J$ is isometrically (triple-)isomorphic to one in the following list: \begin{equation}\label{eq subtriples generated by two min trip} \mathbb{C}, \mathbb{C}\oplus^{\infty} \mathbb{C}, M_{1,2} (\mathbb{C}), M_{2} (\mathbb{C})\hbox{ and } S_2 (\mathbb{C}),
 \end{equation}where $M_{k,n}(\mathbb{C})$ is the C$^*$-algebra of all $k\times n$ matrices with complex entries, and $S_2 (\mathbb{C})$ denotes the symmetric $2\times 2$ complex matrices.\smallskip

Given a minimal tripotent $u$ in a JB$^*$-triple $E$, since $E_2 (u) = \mathbb{C} u$, there exists a unique norm-one functional $\phi_u : E \to \mathbb{C} u$ satisfying $P_2 (u) (x)= \phi_{u} (x) u$,  for each $x$ in $E$. Clearly, $P^{1} (u) (x) = \Re\hbox{e}\phi_u (x) u$ for all $x$ in $E$. By an slight abuse of notation, when no confusion arises, we shall also write $P^{1} (u) (x)$ for the real number $\phi_{u} (x)$.\smallskip

In our next result we shall establish a formula to compute the distance between two minimal tripotents in a JB$^*$-triple.

\begin{proposition}\label{p distance between minimal tripotents} Let $u$ and $v$ be minimal tripotents in a JB$^*$-triple $E$. Then the following formula holds \begin{equation}\label{eq Fla distance between minimal trip} \|u-v\|^2 =  {(1 - \Re\hbox{e} \phi_u (v))+ \sqrt{(1 -   \Re\hbox{e} \phi_u (v))^2 -  \|P_0(u)(v)\|^2}}
\end{equation} $$ =  {(1 - P^{1} (u) (v))+ \sqrt{(1 -   P^{1} (u) (v))^2 -  \|P_0(u)(v)\|^2}}.$$
\end{proposition}

\begin{proof} Let $J$ denote the JB$^*$-subtriple generated by $u$ and $v$. Suppose $J =\mathbb{C}\oplus^{\infty} \mathbb{C}$. The minimality of $u$ and $v$ implies that $u\perp v$ and hence the conclusion in \eqref{eq Fla distance between minimal trip} is a consequence of \cite[Lemma 1.3 $(a)$]{FriRu85}.\smallskip

If $J= \mathbb{C}$ the statement is clear.\smallskip

We assume now that $J =M_{1,2} (\mathbb{C})$. There is no loss of generality in assuming that $u=(1,0)$ and $v=(\lambda_1,\lambda_2)$ with $|\lambda_1|^2 + |\lambda_2|^2 =1$. Therefore, $$\|u-v\|^2 = \| (1-\lambda_1 ,-\lambda_2)\|^2 = |1-\lambda_1|^2 + |\lambda_2|^2 =2 -2 \Re\hbox{e} (\lambda_1) = 2 (1 -  \Re\hbox{e} \phi_u (v)),$$ which proves \eqref{eq Fla distance between minimal trip} because $P_0(u) (v) =0$.\smallskip

We deal now with the case $J=M_{2} (\mathbb{C})$ with the spectral norm. We may assume that $u = \left(
                                                                              \begin{array}{cc}
                                                                                1 & 0 \\
                                                                                0 & 0 \\
                                                                              \end{array}
                                                                            \right)$
and $v = \left(
                                                                              \begin{array}{cc}
                                                                                \overline{\xi_1} \eta_1 & \overline{\xi_2} \eta_1 \\
                                                                                \overline{\xi_1} \eta_2 & \overline{\xi_2} \eta_2 \\
                                                                              \end{array}
                                                                            \right) =(\eta_1,\eta_2)\otimes (\xi_1,\xi_2),$
with $\xi_j,\eta_k\in \mathbb{C}$ satisfying $|\xi_1|^2 + |\xi_2|^2 =1$ and $|\eta_1|^2 + |\eta_2|^2 =1$.  To simplify the notation we write $v= \left(                                                                                        \begin{array}{cc}
 \alpha & \beta \\                                                                                          \gamma & \delta \\                                                                                \end{array}                                                                              \right)$, with $\alpha {\delta} = \beta {\gamma}$ and $|\alpha|^2 + |\delta|^2  + |\beta|^2 + | \gamma|^2 =1.$ According to this terminology, $v-u = \left(                                                                                        \begin{array}{cc}
 \alpha -1 & \beta \\                                                                                          \gamma & \delta \\                                                                                \end{array}                                                                              \right).$ By the Gelfand-Naimark axiom we have $$\|v-u\|^2 = \| (v-u)(v-u)^*\| = \left\| \left(                                                                                        \begin{array}{cc}
 \alpha -1 & \beta \\                                                                                          \gamma & \delta \\                                                                                \end{array}                                                                              \right) \left(                                                                                        \begin{array}{cc}
 \alpha -1 & \beta \\                                                                                          \gamma & \delta \\                                                                                \end{array}                                                                              \right)^* \right\| .$$ Since $(v-u)(v-u)^* = \left(
                                                                           \begin{array}{cc}
                                                                             |\alpha-1|^2 +|\beta|^2 & (\alpha -1) \overline{\gamma} + \beta\overline{\delta} \\
                                                                             (\overline{\alpha} -1) {\gamma} + \overline{\beta}{\delta} & |\gamma|^2 + |\delta|^2 \\
                                                                           \end{array}
                                                                         \right),
 $ its characteristic polynomial is precisely $$ p (\lambda) = \lambda^2 - (|\alpha-1|^2 +|\beta|^2 + |\gamma|^2 + |\delta|^2) \lambda + (|\alpha-1|^2 +|\beta|^2 )( |\gamma|^2 + |\delta|^2) - |(\alpha -1) \overline{\gamma} + \beta\overline{\delta}|^2.$$ The conditions $|\alpha|^2 + |\delta|^2  + |\beta|^2 + | \gamma|^2 =1$ and $\alpha {\delta} = \beta {\gamma}$ imply $$ (|\alpha-1|^2 +|\beta|^2 + |\gamma|^2 + |\delta|^2) = 2- 2 \Re\hbox{e} (\alpha),$$ and $$ (|\alpha-1|^2 +|\beta|^2 )( |\gamma|^2 + |\delta|^2) - |(\alpha -1) \overline{\gamma} + \beta\overline{\delta}|^2 $$ $$= |\alpha-1|^2 |\gamma|^2 + |\alpha-1|^2 |\delta|^2 + |\beta|^2 |\gamma|^2  + |\beta|^2 |\delta|^2 - |\alpha -1|^2 |\gamma|^2 - |\beta|^2 |{\delta}|^2 - 2 \Re\hbox{e} ((\alpha -1) \overline{\gamma} \overline{\beta} {\delta}) $$
 $$=  |\alpha-1|^2 |{\delta}|^2 + | {\beta}|^2 | {\gamma}|^2   - 2 \Re\hbox{e} ((\alpha -1) \overline{\gamma} \overline{\beta} {\delta}) = | (\alpha-1) {\delta} - {\beta} { \gamma} |^2 = | {\delta} |^2. $$ Therefore, $$ p (\lambda) = \lambda^2 - 2(1-  \Re\hbox{e} (\alpha)) \lambda + |\delta|^2,$$ and hence $$\|u-v\|^2 = (1-  \Re\hbox{e} (\alpha)) +\sqrt{ (1-  \Re\hbox{e} (\alpha))^2 - |\delta|^2} $$ $$=  {(1 - \Re\hbox{e} \phi_u (v))+ \sqrt{(1 -   \Re\hbox{e} \phi_u (v))^2 -  \|P_0(u)(v)\|^2}}.$$

The case $J=S_2 (\mathbb{C})$ follows by the same arguments.
\end{proof}

\begin{corollary}\label{c minimal tripotents at distance sqrt2} Let $u$ and $v$ be minimal tripotents in a JB$^*$-triple $E$. Then the condition $\|u\pm v\|= \sqrt{2}$  implies $v\in E^{-1} (u) \oplus E_1 (u).$
\end{corollary}

\begin{proof} Applying Proposition \ref{p distance between minimal tripotents} we know that $$ 2= \|u\mp v\|^2 = {(1 \mp \Re\hbox{e} \phi_u (v))+ \sqrt{(1 \mp  \Re\hbox{e} \phi_u (v))^2 -  \|P_0(u)(v)\|^2}},$$ and hence $$ {(1 \pm \Re\hbox{e} \phi_u (v))= \sqrt{(1 \mp  \Re\hbox{e} \phi_u (v))^2 -  \|P_0(u)(v)\|^2}},$$ $$ {(1 \pm \Re\hbox{e} \phi_u (v))^2=  (1 \mp  \Re\hbox{e} \phi_u (v))^2 -  \|P_0(u)(v)\|^2},$$ $$ \pm 4 \Re\hbox{e} \phi_u (v) = -  \|P_0(u)(v)\|^2,$$ which proves $ \Re\hbox{e} \phi_u (v) = P_0(u)(v)=0,$ and gives the desired conclusion.
\end{proof}

We can now prove that surjective isometries between the unit spheres of elementary JB$^*$-triples of rank $\geq 2$ preserve collinearity between minimal tripotents.

\begin{proposition}\label{p collinear are mapped into collinear} Let $f: S(C) \to S(C^{\prime})$ be a surjective isometry between elementary JB$^*$-triples with rank$(C)\geq 2$. Suppose $e_1$ and $e_2$ are minimal tripotents in $C$ with $e_1\top e_2$. Then $f(e_1) \top f(e_2)$.
\end{proposition}

\begin{proof} Since rank$(C)\geq 2$ by Lemma \ref{l technical headache} there exists a minimal tripotent $u$ in $C$ satisfying  $e_1\perp u$ and $e_2 \top u$. By Proposition \ref{p surjective isometries between the spheres preserve finite rank tripotents} $f(e_1),$ $f(\pm e_2),$ and $f(u)$ are minimal tripotents with $f(e_1)\perp f(u).$  By hypothesis and Theorem \ref{t Tingley antipodes for finite rank new} we have  $$2= \|e_1 \pm e_2\|^2 = \| u\pm e_2\|^2 = \|f(e_1) \pm f(e_2)\|^2 = \| f(u)\pm f(e_2)\|^2.$$ Thus, Corollary \ref{c minimal tripotents at distance sqrt2} shows that $f(e_2) = P^{-1} (f(u)) (f(e_2)) + P_1 (f(u)) (f(e_2))$ and $f(e_2) = P^{-1} (f(e_1)) (f(e_2)) + P_1 (f(e_1)) (f(e_2)).$ Having in mind that $f(e_1) \perp f(u)$, and hence $(C^\prime)^{-1} (f(e_1))\subseteq (C^\prime)_{0} (f(u))$, we obtain $f(e_2) =  P_1 (f(u)) (f(e_2)).$ Lemma \ref{l collinearity}$(b)$ gives  $f(e_1) \top f(e_2)$.
\end{proof}

The next result determines the behavior of a surjective isometry on the space $C^{-1} (e)=i \mathbb{R} e$ associated with a minimal tripotent $e$.

\begin{lemma}\label{l i times minimal trip} Let $f: S(C) \to S(C^{\prime})$ be a surjective isometry between elementary JB$^*$-triples with rank$(C)\geq 2$. Suppose $e$ is a minimal tripotent in $C$. Then $f( i e) = \pm i  f(e)$. Consequently, $f(\lambda e) =\lambda f(e)$ or $f(\lambda e) =\overline{\lambda} f(e)$ for every $\lambda\in \mathbb{C}$.
\end{lemma}

\begin{proof} The elements $f(e)$ and $f(i e)$ are minimal tripotents in $C^\prime$ (see Proposition \ref{p surjective isometries between the spheres preserve finite rank tripotents}). Since, by Theorem \ref{t Tingley antipodes for finite rank new}, we have $$\|f(e)\pm f(ie) \| =\|f(e) - f(\mp ie) \| = \|e \pm i e\| =\sqrt{2},$$ we deduce from Corollary \ref{c minimal tripotents at distance sqrt2} that $f(ie) \in C^{-1} (f(e)) \oplus C_{1} (f(e))= i \mathbb{R} f(e) \oplus C_{1} (f(e))$. If we consider the JB$^*$-subtriple $J$ generated by $f(e)$ and $f(ie)$, we know that $J$ is isomorphic to one of $\mathbb{C}, $ $\mathbb{C}\oplus^{\infty} \mathbb{C},$ $M_{1,2} (\mathbb{C}),$ $M_{2} (\mathbb{C})$ and $S_2 (\mathbb{C})$ (see \cite[Proposition 5]{FriRu85}). Arguing case by case, we can find a minimal tripotent $v_{12}\in C^\prime$, a real $t$ and a complex $\lambda_{12}$ satisfying $t^2 + |\lambda_{12}|^2 =1$, $f(e)\top v_{12}$, and  $f(ie) = i t e + \lambda_{12} v_{12}.$\smallskip

Applying Lemma \ref{l technical headache} we find a minimal tripotent $v_{22}$ in $C^\prime$ such that $f(e)\perp v_{22}$ and $v_{12}\top v_{22}$. Theorem \ref{t Tingley antipodes for finite rank new} applied to $f^{-1}$ implies that $e\perp f^{-1} (v_{22})$ and $f^{-1} (v_{22})$ is a minimal tripotent in $C$. Therefore $i e\perp f^{-1} (v_{22})$. A new application of Theorem \ref{t Tingley antipodes for finite rank new}, shows that $i t + \lambda_{12} v_{12}= f(ie) \perp v_{22}$, which implies $\lambda_{12} =0$, and consequently, $f(ie ) =\pm i f(e)$.\smallskip

To prove the last statement, let us take $u$ in $C$ with $e\perp u$ and let $T_u : C_0(u) \to C^\prime_0(f(e))$ be the surjective real linear isometry given by Proposition \ref{p surjective isometries between the spheres preserve finite rank tripotents}. For each complex number $\lambda$, we have $$f(\lambda e) = T_u(\lambda u) = \Re\hbox{e} (\lambda) T_u(e) + \Im\hbox{m} (\lambda) T_u( i e) $$ $$=  \Re\hbox{e} (\lambda) f(e) + \Im\hbox{m} (\lambda) f( i e) =\Re\hbox{e} (\lambda) f(e) \pm i \Im\hbox{m} (\lambda) f( e),$$ that is $f(\lambda e) =\lambda f(e)$ or $f(\lambda e) =\overline{\lambda} f(e)$, for every $\lambda\in \mathbb{C}$.
\end{proof}

Following the notation in \cite{DanFri87} (see also \cite{Neher87}), we recall that an ordered quadruple $(u_{1},u_{2},u_{3},u_{4})$ of tripotents in a JB$^*$-triple $E$ is called a \emph{quadrangle} if $u_{1}\bot u_{3}$, $u_{2}\bot u_{4}$, $u_{1}\top u_{2}$ $\top u_{3}\top u_{4}$ $\top u_{1}$ and $u_{4}=2 \J
{u_{1}}{u_{2}}{u_{3}}$ (the axiom (JB$^*$1) implies that the last equality holds if the indices are permutated
cyclically, e.g. $u_{2} = 2 \{{u_{3}},{u_{4}},{u_{1}}\}$).\smallskip

Let $u$ and $v$ be tripotents in $E$. We say that $u$ \emph{governs} $v$, $u \vdash v$, whenever $v\in U_{2} (u)$ and $u\in U_{1} (v)$.
An ordered triplet $ (v,u,\tilde v)$ of tripotents in $E$, is called a \emph{trangle} if $v\bot \tilde v$, $u\vdash v$, $u\vdash \tilde v$ and $ v = Q(u)\tilde v$.

\begin{proposition}\label{p quadrangles and trangles} Let $f: S(C) \to S(C^{\prime})$ be a surjective isometry between elementary JB$^*$-triples with rank$(C)\geq 2$, and let $e$ be a minimal tripotent in $C$. The following statements hold:\begin{enumerate}[$(a)$]\item If $(u_{1},u_{2},u_{3},u_{4})$ is a quadrangle of minimal tripotents in $C$, then the quadruple $(f(u_{1}),f(u_{2}),f(u_{3}),f(u_{4}))$ is a quadrangle of minimal tripotents in $C^\prime$;
\item Suppose $(u_{1},u_{2},u_{3},u_{4})$ is a quadrangle of minimal tripotents in $C$ and $f(i u_1) = i f(u_1)$ {\rm(}respectively, $f(i u_1) = -i f(u_1)${\rm)}, then $f(i u_j) = i f(u_j)$ {\rm(}respectively, $f(i u_j) = -i f(u_j)${\rm)}, for every $j=2,3,4$;
\item Suppose $f(i e) = i f(e)$ {\rm(}respectively, $f(i e) = -i f(e)${\rm)}, then $f(i v) = i f(v)$ {\rm(}respectively, $f(i v) = -i f(v)${\rm)} for every minimal tripotent $v\in C$ with $v\top e$;
\item If $ (v,u,\tilde v)$ is a trangle in $C$ with $v, \tilde v$ minimal, then $ (f(v),f(u),f(\tilde v))$ is a trangle in $C^\prime$, with $f(v)$ and $f(\tilde v)$ minimal;
\item Suppose $ (v,u,\tilde v)$ is a trangle in $C$ with $v, \tilde v$ minimal. If $f(i v) = i f(v)$  {\rm(}respectively, $f(i v) = -i f(v)${\rm)}, then $f(i \tilde v) = i f(\tilde v)$ and $f(i u) = i f(u)$ {\rm(}respectively, $f(i \tilde v) = -i f(\tilde v)$ and $f(i u) = -i f(u)${\rm)};
\item Suppose $ (v,u,\tilde v)$ is a trangle in $C$ with $v, \tilde v$ minimal. If $f(i u) = i f(u)$  {\rm(}respectively, $f(i u) = -i f(u)${\rm)}, then $f(i \tilde v) = i f(\tilde v)$ and $f(i v) = i f(v)$ {\rm(}respectively, $f(i \tilde v) = -i f(\tilde v)$ and $f(i v) = -i f(v)${\rm)}.
\end{enumerate}
\end{proposition}

\begin{proof}$(a)$ We know from Theorem \ref{t Tingley antipodes for finite rank new} and Proposition \ref{p collinear are mapped into collinear} that $f(u_{1}),f(u_{2}),$ $f(u_{3}),$ and $f(u_{4})$ are minimal tripotents in $C^\prime$ with $f(u_{1})\top f(u_{2})\top f(u_{3}) \top f(u_{4})\top f(u_{1}),$ $f(u_{1})\perp f(u_{3})$ and $f(u_{2})\perp f(u_{4})$. We only have to show that $$2 \{f(u_{1}), f(u_{2}), f(u_{3})  \} = f(u_{4})$$ to conclude the proof. To this end we observe that $w = \frac12 (u_{1}+u_{2}+u_{3}+u_{4})$ and $\tilde w = \frac12 ( u_{1}-u_{2}+u_{3}-u_{4})$ are minimal tripotents in $C$ with $w\perp \tilde w$. Theorem \ref{t Tingley antipodes for finite rank new} implies that \begin{equation}\label{eq 1711 a} f (w) + f(\tilde w) = f (w + \tilde w) = f(u_{1}+u_{3}) = f(u_{1})+f(u_{3}),
 \end{equation} which implies that $f (w)$ and $f(\tilde w)$ are orthogonal projections in the JBW$^*$-algebra $C^\prime_2 (f(u_{1})+f(u_{3}))$. To see this, we simply observe that $\{f(w),f(u_{1})+f(u_{3}),f(w)\} =$ $ \{f(w),f(w),f(w)\} $ $= f(w) =$ $ \{f(u_{1})+f(u_{3}),f(w),f(u_{1})+f(u_{3})\},$ and the same for $f(\tilde w)$.\smallskip

In particular, \begin{equation}\label{eq 1711 b} f(w) -f(\tilde w)= f(u_{2}+u_{4}) = f(u_{2})+f(u_{4})
 \end{equation} is a symmetry in the JBW$^*$-algebra $C^\prime_2 (f(u_{1})+f(u_{3}))$ and hence $$f(u_{2})+f(u_{4}) = \{f(u_{1})+f(u_{3}),f(u_{2})+f(u_{4}), f(u_{1})+f(u_{3})\},$$ which via Peirce arithmetic, shows that $$ f(u_{2})+f(u_{4}) = 2 \{f(u_{1}) ,f(u_{2}) , f(u_{3})\} + 2 \{f(u_{1}) , f(u_{4}), f(u_{3})\},$$ and hence $2 \{f(u_{1}) ,f(u_{2}) , f(u_{3})\} = f(u_{4})$.\smallskip

$(b)$ Let us assume that $f(i u_1) = i f(u_1)$. By $(a)$, $(f(u_{1}),f(u_{2}),f(u_{3}),f(u_{4}))$ is a quadrangle of minimal tripotents in $C^\prime$, and for the minimal tripotents $w = \frac12 (u_{1}+u_{2}+u_{3}+u_{4})$ and $\tilde w = \frac12 ( u_{1}-u_{2}+u_{3}-u_{4})$, by \eqref{eq 1711 a} and \eqref{eq 1711 b}, we have $f(w) = \frac12 (f(u_{1})+f(u_{2})+f(u_{3})+f(u_{4}))$ and $f(\tilde w) = \frac12 (f(u_{1})-f(u_{2})+f(u_{3})-f(u_{4})).$ By Lemma \ref{l i times minimal trip} it follows that $f(i(w+\tilde w))\in \{ \pm i f(w\pm\tilde w)\}$. Therefore, applying Theorem \ref{t Tingley antipodes for finite rank new}, we obtain
$$i f(u_1) + f( i u_3) =  f(i u_1) + f(i u_3) =  f(i u_1  +i  u_3) = f(i(w+\tilde w))$$ $$ \in\{ \pm i f(w\pm \tilde w)\} = \{\pm i f(u_1+ u_3),\pm i f(u_2+ u_4)\} $$ $$= \{ \pm i (f(u_1) + f(u_3)), \pm i (f(u_2) + f(u_4))\} ,$$ and since $f(i u_3 ) \in \{\pm i f(u_3)\}$, the unique possible choice for $f( i u_3)$ is $i f(u_3)$. By orthogonality $f(i w) + f( i \tilde w) = f( i w + i \tilde w)  =  f(i u_1  +i  u_3)= f(i u_1)  + f( i u_3) = i f( u_1)  + i f( u_3) = i (f(w) +f(\tilde w))$, and since $f(i w)\in \{\pm i f( w)\}$ and $f(i \tilde w)\in \{\pm i f(\tilde w)\}$, we obtain $f(i w) = i f(w)$ and $f(i \tilde w) = i f( \tilde w)$. Consequently, $f(i u_2 + i u_4) = f(i  w - i \tilde w) = f(i  w ) - f(i \tilde w) = i f(w) - i f(\tilde w) = i f(u_2) + i f(u_4),$ which implies $ f(i u_2) = i f(u_2)$ and $f(i u_4) = i f(u_4)$.\smallskip

$(c)$ Let $v$ be a minimal tripotent in $C$ with $e\top v$. By Lemma \ref{l technical headache} there exists a minimal tripotent $u\in C$ such that $e\perp u$ and $v\top u$. By \cite[Proposition 1.7]{DanFri87}, the element $\tilde v = 2 \{ e, v, u\}$ is a minimal tripotent in $C$ and $(e,v,u,\tilde v)$ is a quadrangle in $C$. Since $f(i e) = i f(e)$, it follows from $(b)$ that $f(i v) = i f(v)$.\smallskip

$(d)$ Let  $ (v,u,\tilde v)$ be a trangle in $C$ with $v, \tilde v$ minimal. It is known that $w = \frac12 (v + \tilde v + u)$ and  $\tilde w = \frac12 (v + \tilde v - u)$ are minimal tripotents in $C$ with $w\perp \tilde w$ (compare, for example, \cite[Theorem 4.10]{Ka81} or \cite[Corollary 2.8 and Remark 2.6]{FerMarPe}). Therefore $u = w -\tilde w$ is a rank 2 tripotent in $C$. It follows from Theorem \ref{t Tingley antipodes for finite rank new} that $f(u)$, $f(v),$ $f(w),$ $f(\tilde w),$ and $f(\tilde v)$ are tripotents in $C^\prime,$ where $f(v),$ $f(w),$ $f(\tilde w),$ and $f(\tilde v)$ are minimal, $f(u)$ has rank 2, $f(v) \perp f(\tilde v)$, $f(w)\perp f(\tilde w),$ and  $f(u)= f(w) - f(\tilde w)$.\smallskip

By the hypothesis on $f$, Theorem \ref{t Tingley antipodes for finite rank new} and Proposition \ref{p distance between minimal tripotents}, we have $$ {(1 + P^{1} (f(v)) (f(w)))+ \sqrt{(1 +  P^{1} (f(v)) (f(w)))^2 -  \|P_0(f(v))(f(w))\|^2}} $$ $$ =\|f(v)+ f(w)\|^2= \| v + w\|^2 = \frac 32+\sqrt{2}$$ and $$  {(1 - P^{1} (f(v)) (f(w)))+ \sqrt{(1 -   P^{1} (f(v)) (f(w)))^2 -  \|P_0(f(v))(f(w))\|^2}}$$ $$=\|f(v)- f(w)\|^2 = \| v - w\|^2 = \frac 12,$$ where we have identified $P^{1} (f(v)) (f(w)))$ with the real number $\Re\hbox{e} (\phi_{f(v)} (f(w)))$. It is not hard to see that the unique solution of the above system gives $$P^{1} (f(v)) (f(w))) = \frac12 f(v),\hbox{ and } \|P_0(f(v))(f(w))\| =\frac12.$$ By replacing $v$ with $\tilde v$ we get
$$P^{1} (f(\tilde v)) (f(w))) = \frac12 f(\tilde v),\hbox{ and } \|P_0(f(\tilde v))(f(w))\| =\frac12.$$ We also know that $v +\tilde v = w +\tilde{w}$ and hence $f(v) + f(\tilde v) = f(w) + f(\tilde w)$, and hence $f(w), f(\tilde w)\in C^\prime_2 (f(v)+f(\tilde v))$. Therefore, $P_0(f(v))(f(w)) = P_2(f(\tilde v))(f(w))$. If we observe that $(C^{\prime})^{1} (f(v))\subseteq C^{\prime}_{2} (f(v)) = \mathbb{C} f(v)$, we obtain $P_2(f(v)) (f(w)) = P^{1} (f(v)) (f(w))) = \frac12 f(v),$ and thus $f(w) = \frac12 f(v) + P_1(f(v)) (f(w)) +\frac12 f(\tilde v)$. It is also clear from the above that $P_1(f(\tilde v)) (f(w)) = P_1(f(v)) (f(w))$. We have therefore shown that $$ f(w) = \frac12 f(v)+ \frac12 f(\tilde v) + P_1(f(\tilde v)) (f(w)).$$ Similar arguments applied to $f(\tilde w)$ prove $$ f(\tilde w) = \frac12 f(v)+ \frac12 f(\tilde v) + P_1(f(\tilde v)) (f(\tilde w)).$$
The equality $f(v) + f(\tilde v) = f(w) + f(\tilde w)$ implies that $$P_1(f(\tilde v)) (f(\tilde w)) = - P_1(f(\tilde v)) (f( w)),$$ and hence $$f(u) = f(w) - f(\tilde w) = 2 P_1(f(\tilde v)) (f( w)) = 2 P_1(f( v)) (f( w))\in C^{\prime}_{1} (f(v))\cap C^{\prime}_{1} (f(\tilde v)).$$

The identity $$\{f(u), f(v) +f(\tilde v), f(u) \} = \{f(w) - f(\tilde w), f(w) + f(\tilde w),f(w) - f(\tilde w) \} $$ $$=\{f(w),f(w),f(w) \}+ \{f(\tilde w),f(\tilde w),f(\tilde w)\} = f(w)+ f(\tilde w) = f(v) +f(\tilde v),$$ proves that $f(u)$ actually is a symmetry in the JB$^*$-algebra  $C^\prime_2 (f(v)+f(\tilde v))$, and then $C^\prime_2 (f(v)+f(\tilde v)) = C^\prime_2 (f(u))$, which gives $f(v),f(\tilde v) \in  C^\prime_2 (f(u))$.\smallskip

Finally, since $f(u) \in C^{\prime}_{1} (f(v))\cap C^{\prime}_{1} (f(\tilde v))$, we deduce from Peirce rules that $$\{f(u), f(v), f(u)\}\in C^\prime_0 (f(v))\cap C^\prime_2 (f(v)+f(\tilde v)) = \mathbb{C} f(\tilde v),$$
$$\{f(u), f(\tilde v), f(u)\}\in C^\prime_0 (f(\tilde v))\cap C^\prime_2 (f(v)+f(\tilde v)) = \mathbb{C} f( v),$$  and $$\{f(u), f(v)+f(\tilde v), f(u)\}= f(v)+f(\tilde v),$$ and hence $\{f(u), f(v), f(u) \} = f(\tilde v)$ and $\{f(u), f(\tilde v), f(u)\}= f(v),$ which finishes the proof of $(d)$.\smallskip

$(e)$ Let $ (v,u,\tilde v)$ be a trangle in $C$ with $v, \tilde v$ minimal, and $f(i v) = i f(v)$. By $(d)$ the 3-tuple $ (f(i v),f(i u), f(i \tilde v))$ is a trangle in $C^\prime$ with $f(i v), f(i \tilde v)$ minimal. Following the arguments in the proof of $(d)$, $ i w =  \frac12 i ( v + \tilde v + u)$,  $i \tilde w = \frac12i  (v + \tilde v - u),$ $f(i w)$ and $f( i \tilde w)$ are minimal tripotents with $i w\perp i \tilde w$, $f(i w) \perp f(i \tilde w)$, $f( u)= f( w) - f( \tilde w),$ $f( v) + f( \tilde v) = f( w) + f( \tilde w),$ $$f(i u)= f(i w) - f(i \tilde w)\in C^{\prime}_{1} (f(i v))\cap C^{\prime}_{1} (f( i \tilde v)),$$ and $$f(i v) + f(i \tilde v) = f(i w) + f(i \tilde w).$$ Since $f(i v) = i f(v)$ and $f(i z)\in \{\pm i f(z)\}$, for every $z= w,\tilde w, \tilde v$, we deduce that $f(i \tilde v) = i f(\tilde v)$ and $f(i u) = i f(u)$.\smallskip

$(f)$ With the notation employed in the proofs of $(d)$ and $(e),$ if $f(i u) = i f(u)$, we have $i (f(w) -f(\tilde w)) = i f(w- \tilde w) = i f(u) = f(i u ) =  f(i w- i \tilde w) = f(i w) -f(i \tilde w)$, and hence, by orthogonality relations, $f(i w) = i f(w)$, $f(i \tilde w) = i f(\tilde w)$. We similarly get $f(i \tilde v) = i f(\tilde v)$ and $f(i v) = i f(v)$.
\end{proof}

We have developed enough tools to establish that surjective isometries between the unit spheres of elementary JB$^*$-triples of rank greater or equal than 2 are $\ell_2$-additive on collinear minimal tripotents.

\begin{proposition}\label{p additivity on collinear min trip} Let $f: S(C) \to S(C^{\prime})$ be a surjective isometry between elementary JB$^*$-triples with rank$(C)\geq 2$. Suppose $e_1$ and $e_2$ are minimal tripotent in $C$ with $e_1\top e_2$. Then, the following statements hold:
\begin{enumerate}[$(a)$]\item If $f(i e_1) =i f(e_1)$, then $$f( \alpha e_1 + \beta e_2 ) = \alpha  f(e_1) + \beta f(e_2),$$ for all $\alpha, \beta \in \mathbb{C}$ with $|\alpha|^2 +|\beta|^2 =1$;
\item If $f(i e_1) = - i f(e_1)$, then $$f( \alpha e_1 + \beta e_2 ) = \overline{\alpha}  f(e_1) + \overline{\beta} f(e_2),$$ for all $\alpha, \beta \in \mathbb{C}$ with $|\alpha|^2 +|\beta|^2 =1$.
\end{enumerate}
\end{proposition}

\begin{proof} Let us fix $\alpha, \beta \in \mathbb{C}$ with $|\alpha|^2 +|\beta|^2 =1$. We can assume $\alpha,\beta\neq 0$. Proposition \ref{p surjective isometries between the spheres preserve finite rank tripotents} assures that $f(e_1),$ $f(e_2)$, and $f( \alpha e_1 + \beta e_2 ) $ are minimal tripotents in $C^\prime$. Let $J$ denote the JB$^*$-subtriple of $C^\prime$ generated by $f(e_1)$ and $f( \alpha e_1 + \beta e_2 ) $. We have already commented that $J$ identifies with one of the following $\mathbb{C}, $ $\mathbb{C}\oplus^{\infty} \mathbb{C},$ $M_{1,2} (\mathbb{C}),$ $M_{2} (\mathbb{C})$ and $S_2 (\mathbb{C})$ (see \cite[Proposition 5]{FriRu85}).\smallskip

If $J=\mathbb{C}$, we have $f( \alpha e_1 + \beta e_2 ) = \lambda f(e_1)$ for a suitable complex $\lambda$ with $|\lambda|=1$. By Lemma \ref{l i times minimal trip} we have $\lambda f(e_1) =  f(\lambda e_1)$ or $\lambda f(e_1) =  f(\overline{\lambda} e_1)$. Therefore, $\alpha e_1 + \beta e_2 =\lambda e_1$ or $\alpha e_1 + \beta e_2 =\overline{\lambda} e_1$, and both equalities are impossible.\smallskip

If $J= \mathbb{C}\oplus^{\infty} \mathbb{C},$ we can assume $f(e_1) =(1,0)$ and $f( \alpha e_1 + \beta e_2 ) = (\lambda,0)$ or $f( \alpha e_1 + \beta e_2 ) = (0,\lambda)$ with $|\lambda|=1$. In the first case, Lemma \ref{l i times minimal trip} gives $\alpha e_1 + \beta e_2 =\lambda e_1$ or $\alpha e_1 + \beta e_2 =\overline{\lambda} e_1$, which is impossible, while in the second case, by Theorem \ref{thm Tyngley co sums with more than one element new}, we have $\alpha e_1 + \beta e_2 \perp e_1$, which is impossible too.\smallskip

In the remaining cases, we can assume $J\subseteq M_{2} (\mathbb{C})$, $f(e_1) = \left(
                                                                                         \begin{array}{cc}
                                                                                           1 & 0 \\
                                                                                           0 & 0 \\
                                                                                         \end{array}
                                                                                       \right)
$ and $f( \alpha e_1 + \beta e_2 ) = \left(
                                       \begin{array}{cc}
                                         \alpha^\prime & \beta^\prime \\
                                         \gamma^\prime & \delta^\prime \\
                                       \end{array}
                                     \right)
$ with $|\alpha^\prime|^2 +| \beta^\prime|^2 + |\gamma^\prime |^2 + |\delta^\prime|^2 =1$ and $\alpha^\prime  \delta^\prime  = \beta^\prime \gamma^\prime$. Applying Proposition \ref{p distance between minimal tripotents} and the properties of $f$ we get:
$$2(1\pm \Re\hbox{e}(\alpha)) = \| e_1 \pm ( \alpha e_1 + \beta e_2 )\|^2 = \| f(e_1) \pm f( \alpha e_1 + \beta e_2 )\|^2 $$ $$ = {(1 \pm \Re\hbox{e}(\alpha^\prime))+ \sqrt{(1 \pm  \Re\hbox{e} (\alpha^\prime))^2 -  |\delta^\prime |^2}};$$ $$ 4 (1\pm \Re\hbox{e}(\alpha))^2 + (1 \pm \Re\hbox{e}(\alpha^\prime))^2- 4(1\pm \Re\hbox{e}(\alpha)) (1 \pm \Re\hbox{e}(\alpha^\prime)) = (1 \pm  \Re\hbox{e} (\alpha^\prime))^2 -  |\delta^\prime |^2 ;$$ $$\pm  4 (1\pm \Re\hbox{e}(\alpha)) (\Re\hbox{e}(\alpha^\prime)-  \Re\hbox{e}(\alpha)) = -  |\delta^\prime |^2 .$$ Therefore, $\delta^\prime =0$, and since $|\alpha|\neq 1$ we also deduce that $\Re\hbox{e}(\alpha^\prime)= \Re\hbox{e}(\alpha)$. Therefore, $P_0 (f(e_1)) (f( \alpha e_1 + \beta e_2 ))=0$ and $P^{1} (f(e_1)) (f( \alpha e_1 + \beta e_2 )) = \Re\hbox{e}(\alpha) f(e_1)$.\smallskip

By Lemma \ref{l i times minimal trip} we have $ f(i e_1)= i f(e_1)$ or $ f( i e_1) =  -i f(e_1)$. We shall distinguish these two cases.\smallskip

\emph{Case $a)$} $ f(i e_1)= i f(e_1)$. By Proposition \ref{p distance between minimal tripotents}, Theorem \ref{t Tingley antipodes for finite rank new} and the hypothesis we get
$$2(1\pm \Re\hbox{e}(i \alpha))  = \|e_1 \pm i  (\alpha e_1 + \beta e_2) \|^2 = \|i e_1 \mp   (\alpha e_1 + \beta e_2) \|^2 = \|f(i e_1) \mp   f (\alpha e_1 + \beta e_2) \|^2 $$ $$=\left\| \left(
                                                                                         \begin{array}{cc}
                                                                                           i & 0 \\
                                                                                           0 & 0 \\
                                                                                         \end{array}
                                                                                       \right) \mp \left(
                                       \begin{array}{cc}
                                         \alpha^\prime & \beta^\prime \\
                                         \gamma^\prime & 0 \\
                                       \end{array}
                                     \right)\right\|^2=\left\| \left(
                                                                                         \begin{array}{cc}
                                                                                           1 & 0 \\
                                                                                           0 & 0 \\
                                                                                         \end{array}
                                                                                       \right) \pm \left(
                                       \begin{array}{cc}
                                         i \alpha^\prime & i \beta^\prime \\
                                         i \gamma^\prime & 0 \\
                                       \end{array}
                                     \right)\right\|^2 = 2(1\pm \Re\hbox{e}(i \alpha^\prime)),$$
which shows that $\alpha =\alpha^\prime$, and hence \begin{equation}\label{eq 1 1010}f( \alpha e_1 + \beta e_2 ) = \alpha f(e_1) + P_1(f(e_1)) (f( \alpha e_1 + \beta e_2 )).
\end{equation}\smallskip

Since $ f(i e_1)= i f(e_1)$ and $e_1\top e_2$, Proposition \ref{p quadrangles and trangles}$(c)$ implies $f(i e_2) = i f(e_2)$. Thus, repeating the above arguments with $e_2$ in the role of $e_1$ we get \begin{equation}\label{eq 2 1010}f(\alpha e_1 + \beta e_2 ) = \beta f(e_2) + P_1(f(e_2)) (f( \alpha e_1 + \beta e_2 )).
\end{equation} By combining \eqref{eq 1 1010} and \eqref{eq 2 1010} we get $$  f(\alpha e_1 + \beta e_2 ) = \alpha f(e_1)  + \beta f(e_2) +  P_1(f(e_2)) P_1(f(e_1)) (f( \alpha e_1 + \beta e_2 )).$$ Since $f(e_1)$ and $f(e_2)$ are collinear minimal tripotents (compare  Proposition \ref{p collinear are mapped into collinear}), $\alpha f(e_1)  + \beta f(e_2)$ is a minimal tripotent in $C^\prime$ (compare \cite[LEMMA in page 306]{DanFri87}). The element $f(\alpha e_1 + \beta e_2 )$ lies in the unit sphere of $C^\prime$ and, by Peirce arithmetic $P_1(f(e_2)) P_1(f(e_1)) (f( \alpha e_1 + \beta e_2 ))\in C_1^{\prime} (\alpha f(e_1)  + \beta f(e_2)),$ and thus  $$P_2 (\alpha f(e_1)  + \beta f(e_2)) (f(\alpha e_1 + \beta e_2 )) = \alpha f(e_1)  + \beta f(e_2).$$ Lemma 1.6 in \cite{FriRu85} proves $f(\alpha e_1 + \beta e_2 ) = \alpha f(e_1)  + \beta f(e_2)$.\smallskip

In the \emph{case $b)$} $ f(i e_1)=- i f(e_1)$, the above arguments prove $f(\alpha e_1 + \beta e_2 ) = \overline{\alpha} f(e_1)  + \overline{\beta} f(e_2)$.
\end{proof}

Let $C$ be a Cartan factor with rank greater or equal than 2. Let $e_1$ be a minimal tripotent in $C$. The Peirce subspace $C_1(e_1)$ cannot be zero, otherwise $C=C_2(e_1)\oplus^{\perp} C_0(e_1)$ would be the direct sum of two orthogonal weak$^*$-closed triple ideals, which is impossible. Applying \cite[Corollary 2.2 and Proposition 2.1]{DanFri87} one of the following statements holds:\begin{enumerate}[$(i)$]\item There exists a minimal tripotent $v$ in $C$ satisfying $e_1\top v$;
\item There exist a rank 2 tripotent $u$ and a minimal tripotent ${\tilde e}_1$ in $C$ such that $(e_1, u, {\tilde e}_1)$ is a trangle;
\item There exist minimal tripotents $v_2,v_3,v_4$ in $C$ such that $(e_1,v_2,v_3,v_4)$ is a quadrangle.
\end{enumerate}

In case $(i)$, we can repeat the argument in the proof of Proposition \ref{p quadrangles and trangles}$(c)$ to deduce, via Lemma \ref{l technical headache}, the existence of minimal tripotents $v_2,v_3,v_4$ in $C$ such that $(e_1,v_2,v_3,v_4)$ is a quadrangle. Therefore, for each minimal tripotent $e_1$ in $C$ one of the following holds: \begin{enumerate}[$(\checkmark.1)$]\label{eq minimal in trangle or quadrangle}
\item There exist a rank 2 tripotent $u$ and a minimal tripotent ${\tilde e}_1$ in $C$ such that $(e_1, u, {\tilde e}_1)$ is a trangle;
\item There exist minimal tripotents $v_2,v_3,v_4$ in $C$ such that $(e_1,v_2,v_3,v_4)$ is a quadrangle.
\end{enumerate}

Let $e_2$ be a minimal tripotent with $e_1\perp e_2$. In each one of the previous cases, by \cite[Proposition 5.8]{Ka97}, there exists a complex linear, isometric, JB$^*$-triple isomorphism $T: C\to C$ such that \begin{enumerate}[$(b.1)$]\label{eq minimals in trangle or quadrangle}
\item $T(e_1)=e_1$ and $T({\tilde e}_1) = e_2$;
\item $T(e_1)=e_1$ and $T(v_3) = e_2$.
\end{enumerate} Since $T$ preserves quadrangles and trangles of the previous form, we can always conclude that one of the following statements is true:\begin{enumerate}[$(c.1)$]\item There exist a rank 2 tripotent $u$ in $C$ such that $(e_1, u, e_2)$ is a trangle;
\item There exist minimal tripotents $v_2,v_4$ in $C$ such that $(e_1,v_2,e_2,v_4)$ is a quadrangle.
\end{enumerate} The following corollary is therefore a consequence of the previous arguments $(c.1)$ and $(c.2)$ and Proposition \ref{p quadrangles and trangles}$(b)$ and $(e)$.

\begin{corollary}\label{c linearity or anti-linearity for all minimal orthogonal} Let $f: S(C) \to S(C^{\prime})$ be a surjective isometry between elementary JB$^*$-triples with rank$(C)\geq 2$, and let $e_1$ and $e_2$ be minimal tripotents in $C$ with $e_1\perp e_2$. Suppose $f(i e_1) = i f(e_1)$ {\rm(}respectively, $f(i e_1) = - i f(e_1)${\rm)}, then $f(i e_2) = i f(e_2)$ {\rm(}respectively, $f(i e_2) = - i f(e_2)${\rm)}.$\hfill\Box$
\end{corollary}

Our next lemma will also follow from the comments prior to Corollary \ref{c linearity or anti-linearity for all minimal orthogonal} and \cite[Proposition 5]{FriRu85}.

\begin{lemma}\label{l FR Prop 5 with a grid} Let $e$ and $v$ be two minimal tripotents in a Cartan factor of rank greater or equal than two. Then one of the following statements holds:\begin{enumerate}[$(a)$]\item There exist minimal tripotents $v_2,v_3,v_4$ in $C$, and complex numbers $\alpha$, $\beta$, $\gamma$, $\delta$ such that $(e,v_2,v_3,v_4)$ is a quadrangle, $|\alpha|^2 +| \beta|^2 + |\gamma|^2 + |\delta|^2 =1$, $\alpha \delta  = \beta \gamma$, and $v = \alpha e + \beta v_2 + \gamma v_4 + \delta v_3$;
\item There exist a minimal tripotent $\tilde e\in C$, a rank two tripotent $u\in C$, and complex numbers $\alpha, \beta, \delta$ such that $(e, u,\tilde e)$ is a trangle, $|\alpha|^2 +2 | \beta|^2 + |\delta|^2 =1$, $\alpha \delta  = \beta^2$, and $v = \alpha e+ \beta u +\delta \tilde e$.
\end{enumerate}
\end{lemma}

\begin{proof} Let $J$ denote the JB$^*$-subtriple of $C$ generated by $e$ and $v$. We have repeatedly applied that $J$ identifies isomorphically with one of the following list: $\mathbb{C}, $ $\mathbb{C}\oplus^{\infty} \mathbb{C},$ $M_{1,2} (\mathbb{C}),$ $M_{2} (\mathbb{C})$ and $S_2 (\mathbb{C})$ (see \cite[Proposition 5]{FriRu85}).\smallskip

Suppose $J=\mathbb{C}$. Clearly $v=\lambda e$ for a suitable complex number $\lambda$ with $|\lambda|=1.$ By $(\checkmark.1)$ and $(\checkmark.2)$, or there exist a rank 2 tripotent $u$ and a minimal tripotent ${\tilde e}$ in $C$ such that $(e, u, {\tilde e})$ is a trangle, or there exist minimal tripotents $v_2,v_3,v_4$ in $C$ such that $(e,v_2,v_3,v_4)$ is a quadrangle. So, the desired conclusion holds with $\alpha =\lambda$, $\beta=\gamma=\delta =0$ (where $\gamma = \beta$ in the case of a trangle).\smallskip

In the case  $J= \mathbb{C}\oplus^{\infty} \mathbb{C},$ we can assume that $e=(1,0)$ and $v=(\lambda,0)$ or $v = (0,\lambda)$ with $|\lambda|=1$. The case $v=(\lambda,0)$ was treated in the previous paragraph. For the second choice, we observe that  $v\perp e$, and hence the statement follows from $(c.1)$ and $(c.2)$ with $\delta = 1,$ $e_1=e,$ $e_2 =v$ and $\alpha = \beta= \gamma=0$ (where $\gamma = \beta$ in the case of a trangle).\smallskip

If $J= M_{1,2} (\mathbb{C})$. We can obviously find a minimal tripotent $v_2\in C$ such that $e\top v_2$ and complex numbers $\alpha,\beta$ satisfying $v = \alpha e +\beta v_2$ and $|\alpha|^2 + | \beta|^2 =1$. Let us take, via Lemma \ref{l technical headache}, a minimal tripotent $v_3$ in $C$ such that $v_3\perp e$ and $v_2 \top v_3$. Setting $v_4 = 2 \{e,v_2,v_3\}$ we define a quadrangle $(e,v_2,v_3,v_4)$ (see \cite[Proposition 1.7]{DanFri87}). The statement $(b)$ holds with $\alpha,\beta$, $\gamma= \delta =0$.\smallskip

We deal now with the remaining cases. There is no loss of generality in assuming $J\subseteq M_{2} (\mathbb{C})$, $e = \left(
                                                                                         \begin{array}{cc}
                                                                                           1 & 0 \\
                                                                                           0 & 0 \\
                                                                                         \end{array}
                                                                                       \right)
$ and $v = \left(
                                       \begin{array}{cc}
                                         \alpha  & \beta  \\
                                         \gamma  & \delta  \\
                                       \end{array}
                                     \right)
$ with $|\alpha |^2 +| \beta |^2 + |\gamma  |^2 + |\delta |^2 =1$ and $\alpha \delta  = \beta \gamma$ (with $\beta=\gamma$ in case $J= S_2 (\mathbb{C})$). We conclude by taking $v_2= \left(
                                                                                         \begin{array}{cc}
                                                                                           0 & 1 \\
                                                                                           0 & 0 \\
                                                                                         \end{array}
                                                                                       \right),$ $v_3 = \left(
                                                                                         \begin{array}{cc}
                                                                                           0 & 0 \\
                                                                                           0 & 1 \\
                                                                                         \end{array}
                                                                                       \right),$ and $v_4= \left(
                                                                                         \begin{array}{cc}
                                                                                           0 & 0 \\
                                                                                           1 & 0 \\
                                                                                         \end{array}
                                                                                       \right)$ or $u = \left(
                                                                                         \begin{array}{cc}
                                                                                           0 & 1 \\
                                                                                           1 & 0 \\
                                                                                         \end{array}
                                                                                       \right)$ and $\tilde e = \left(
                                                                                         \begin{array}{cc}
                                                                                           0 & 0 \\
                                                                                           0 & 1 \\
                                                                                         \end{array}
                                                                                       \right),$ respectively.
\end{proof}

We recall that a spin factor is a complex Hilbert space $X$, with inner product $(.|.)$, provided
with a conjugation (i.e. a conjugate linear isometry of period 2 for the Hilbertian norm given by $\|x\|_2^2 = (x|x)$ ($x\in X$))
$x\mapsto \overline{x},$ where triple product and norm are given by
\begin{equation}\label{eq spin product}
\{x, y, z\} = (x|y)z + (z|y) x -(x|\overline{z})\overline{y},
\end{equation} and $ \|x\|^2 = (x|x) + \sqrt{(x|x)^2 -|
(x|\overline{x}) |^2},$ respectively.\smallskip

Let $X_1 =\{x\in X : x=\overline{x}\}$ and $X_2 =\{x\in X : x=-\overline{x}\}$. It is not hard to see that $X_1$ and $X_2$ are real subspaces of $X$, $X_2 = i X_1$, and $X = X_1\oplus X_2$. Since $\overline{.}$ is a conjugation we can easily see that $(x| y ) = (\overline{y} | \overline{x})$ for all $x,y\in X$. Therefore, if $x_1,y_1 \in X_1$ and $x_2,y_2\in X_2$ we have $(x_1 | x_2) = - (x_2 | x_1 ) =\overline{-(x_1 | x_2)},$ $(x_1| y_1) = (y_1| x_1) =\overline{ (x_1| y_1)},$  and $(x_2| y_2) = (y_2| x_2)= \overline{ (x_2| y_2)}.$ Therefore, $(X_j | X_j) \subseteq \mathbb{R}$ and $(x_1 | x_2)\in i \mathbb{R}.$  The underlying real Banach space $X_{\mathbb{R}}$ is a real Hilbert space with respect to the inner product $\langle x | y\rangle := \Re\hbox{e} (x| y)$. Clearly, the real subspaces $X_1$ and $X_2$ are orthogonal with respect to the inner product $\langle . | .\rangle$, that is, $\langle X_1 | X_2\rangle =0$, and $\langle x_j | y_j\rangle = (x_j | y_j),$ for every $j=1,2$.\smallskip

For $x_1\in X_1$ and $x_2\in X_2$, we have $\overline{x_1+ x_2} = x_1 -x_2$ and if $(x_1| x_2) =0$ we also have \begin{equation}\label{ eq X1 and X2 are L-sum} \| x_1 + x_2 \| = \|x_1\| + \|x_2\|.
\end{equation} It is known that every spin factor $X$ has rank two. We further known the precise form of minimal and rank two tripotents in $X$, more concretely, $$ \hbox{min } \mathcal{U} (X) =\Big\{ \frac12 (x_1+x_2) : x_i\in S(X_i), (x_1|x_2)=0 \Big\}$$ and $$ \hbox{max } \mathcal{U} (X) =\Big\{ \lambda x_1 : x_1\in S(X_1), \lambda \in S(\mathbb{C}) \Big\}.$$ Every maximal or complete tripotent in $X$ is unitary. Given a minimal tripotent $e =\frac12 (x_1+x_2) \in \hbox{min } \mathcal{U} (X) $, its Peirce-0 subspace \begin{equation}\label{eq Peirce zero in spin is one dimensional} X_0 (e) = \mathbb{C} \overline{e} = \{x\in X : x\perp e\}
\end{equation} is one-dimensional.\smallskip

Let $v= \frac12 (x_1+i x_2)$  ($x_i\in S(X_1)$, $(x_1|x_2)=0$) be a minimal tripotent in $X$. It is easy to check that $$X_2 (v) = \mathbb{C} v, \ X_0 (v) = \mathbb{C} \tilde v,\hbox{ and } X_1 (v) = \{x\in X : (x|x_1) = (x|x_2)=0\}=\{x_1,x_2\}^{\perp}_{X}.$$ We further know that \begin{equation}\label{eq Peirce projections spin} P_2 (v) (x) = 2 (x|v) v = ((x|x_1)-i (x|x_2)) v,
\end{equation} $$P_0 (v) (x) = 2 (x|\tilde v) \tilde v = ((x|x_1)+i (x|x_2)) \tilde v,$$ and  $$P_1(v) (x) = x - 2 (x|v) v- 2 (x|\tilde v) \tilde v =x -  (x|x_1) x_1- (x|x_2) x_2 \ (x\in X).$$ The projection $P_1(v)$ also coincides with the orthogonal projection of $X$ onto $\{x_1,x_2\}_{X}^{\perp}$ in the Hilbert space $(X, (.|.)).$

\begin{lemma}\label{l new Spin} Let $(v,u,\tilde v)$ be a trangle of tripotents in a Cartan factor $C$, where $v$ and $\tilde v$ are minimal. Let $w=\frac12 (v+u+\tilde v)$, $w=\frac12 (v-u+\tilde v)$, $\widehat{u}=v-\tilde v$. Suppose $\alpha, \beta, \delta$ are complex numbers with $|\alpha|^2 +2 | \beta|^2 + |\delta|^2 =1$, and $\alpha \delta  =\beta^2$. Let $x$ be an element in $C$ such that $\|x\|\leq 1$, $$P_2 (v) (x)=\alpha v,\  P_2 (\tilde v) (x)=\delta \tilde v, \ P_2 (w) (x) = \frac{\alpha + 2 \beta + \delta}{2} w,$$ $$\hbox{ and } P_2(\tilde w) (x) = \frac{\alpha - 2 \beta + \delta}{2} \tilde w.$$ Then, for the minimal tripotent $e=\alpha v +\beta u + \delta \tilde v$, we have $x = e + P_0(e) (x)$.
\end{lemma}

\begin{proof} By \cite[Theorem 4·10]{Ka81} (see also \cite[Lemma 2.7]{FerMarPe}), $C_2 (v+\tilde v)$ is (isometrically isomorphic to) a spin factor. Let $X$ denote this spin factor $C_2 (v+\tilde v) = C_2(v) \oplus C_2(\tilde v) \oplus C_1(v)\cap C_1(\tilde v).$ Let us observe that $(w,\widehat{u}, \tilde w)$ is a trangle in $C$ with $w,\tilde w$ minimal. 
\smallskip

Let $\overline{\ \cdot \ }$ and $(.|.)$ denote the involution and the inner product of $X$, respectively. We shall keep the notation given before this lemma.\smallskip

Since $v,\tilde v, u\in C_2(v+\tilde v) =X$ we can assume that $v= \frac12 (x_1+i x_2),$ $\tilde v= \frac12 (x_1-i x_2)=\overline{v}$, and $u = i x_3,$ where $x_i\in S(X_1)$, $(x_1|x_2)=0$, $(x_1|x_3)=0$, and $(x_2|x_3)=0$.\smallskip

Let $y = P_2(v+\tilde v) (x)$. Clearly $\|y\|\leq 1$. By hypothesis $$P_2 (v) (y)=\alpha v,\  P_2 (\tilde v) (y)=\delta \tilde v, \ P_2 (w) (y)  = \frac{\alpha + 2 \beta + \delta}{2} w,$$ $$\hbox{ and } P_2(\tilde w) (y) = \frac{\alpha - 2 \beta + \delta}{2} \tilde w.$$ Applying the identities in \eqref{eq Peirce projections spin}, we deduce from the last four equalities that \begin{equation}\label{eq coef in the Hilbert 3 dimensional x1 x2 x3} (y| x_1) =\frac{\alpha+\delta}{2}, \  (y| x_2) =i \frac{\alpha-\delta}{2}, \hbox{ and } (y| x_3) =i \beta.
\end{equation}

Let $H$ be the (complex) subspace of $X$ generated by $x_1,x_2$ and $x_3$. And let $P: X\to H$ be the orthogonal projection of the Hilbert space $(X, (.|.)$ onto $H$. Since $\overline{H} = H$, it follows from \cite[Remark 7]{JamPeSiTah2014Ceby} that $$\max\{\| P(z) \|, \|(I-P)(z)\|\} \leq \|z\|,$$ for every $z\in X$. Moreover, $\| z \| = \|P(z)\|$ if and only if $z=P(z)$.\smallskip

We have shown in \eqref{eq coef in the Hilbert 3 dimensional x1 x2 x3} that $$P (y) =\frac{\alpha+\delta}{2} x_1 + i \frac{\alpha-\delta}{2} x_2 + i \beta x_3= \alpha v +\beta u + \delta \tilde v.$$ Since, by the hypothesis on $\alpha,\beta,\delta$ we have $$1\geq \|y\| \geq \|P(y)\|= \|\alpha v +\beta u + \delta \tilde v\| = 1,$$ we conclude that $P(y) =y$, and hence $P_2(v+\tilde v) (x) =y= \alpha v +\beta u + \delta \tilde v$.\smallskip

Finally, the element $e= \alpha v +\beta u + \delta \tilde v$ is a minimal tripotent in the spin factor $X=C_2 (v+\tilde v)$ with $P_2(e) (x) = P_2 (e) (y) = e$. The conditions $1\geq \|x\|$, $P_2 (e) (x) = e$ imply, via \cite[Lemma 1.6]{FriRu85}, that $P_1 (e) (x) =0,$ and hence $x= e +P_0(e) (x)$.  \end{proof}

Our next theorem contains a key technical theorem needed for the main results of this note.\smallskip

\begin{theorem}\label{p additivity on trangles and quadrangles} Let $f: S(C) \to S(C^{\prime})$ be a surjective isometry between elementary JB$^*$-triples with rank$(C)\geq 2$. The following statements hold:\begin{enumerate}[$(a)$]\item If $(v_{1},v_{2},v_{3},v_{4})$ is a quadrangle of minimal tripotents in $C$ and $f(i v_1) = i f(v_1)$ {\rm(}respectively, $f(i v_1) = -i f(v_1)${\rm)}, then  $$f(\alpha v_1 + \beta v_2 + \gamma v_4 + \delta v_3)  = \alpha f(v_1) + \beta f(v_2) + \gamma f(v_4) + \delta f(v_3)$$ {\rm(}respectively, $$f(\alpha v_1 + \beta v_2 + \gamma v_4 + \delta v_3)  = \overline{\alpha} f(v_1) + \overline{\beta} f(v_2) + \overline{\gamma} f(v_4) + \overline{\delta} f(v_3)),$$ for all $\alpha, \beta, \gamma, \delta\in \mathbb{C}$ with $|\alpha|^2 +| \beta|^2 + |\gamma|^2 + |\delta|^2 =1$, $\alpha \delta  =\beta \gamma$;
\item If $(v, u,\tilde v)$ is a trangle, with $ v, \tilde{v} \in C$  minimal tripotents, $u\in C$ a rank two tripotent, and $f(i v) = i f(v)$ {\rm(}respectively, $f(i v) = -i f(v)${\rm)}, then $$f(\alpha v + \beta u + \delta \tilde v)  = \alpha f(v) + \beta f(u) + \delta f(\tilde v)$$ {\rm(}respectively, $$f(\alpha v + \beta u + \delta \tilde v)  = \overline{\alpha} f(v) + \overline{\beta} f(u) + \overline{\delta} f(\tilde v){\rm)},$$
    for all  $\alpha, \beta, \delta \in \mathbb{C}$ with  $|\alpha|^2 +2 | \beta|^2 + |\delta|^2 =1$, $\alpha \delta  =\beta^2$.
\end{enumerate}
\end{theorem}

\begin{proof} $(a)$ We assume $f(i v_1) = i f(v_1)$ (the case $f(i v_1) = - i f(v_1)$ follows similarly). Let $e=\alpha v_1 + \beta v_2 + \gamma v_4 + \delta v_3$. By Theorem \ref{t Tingley antipodes for finite rank new}, $f(e)$ is a minimal tripotent and $f(-e) =-f(e)$. By Lemma \ref{l FR Prop 5 with a grid} one of the following statements holds:\begin{enumerate}[$(1)$]\item There exist minimal tripotents $w_2,w_3,w_4$ in $C^{\prime}$, and complex numbers $\alpha'$, $\beta'$, $\gamma'$, $\delta'$ such that $(f(v_1),w_2,w_3,w_4)$ is a quadrangle, $|\alpha'|^2 +| \beta'|^2 + |\gamma'|^2 + |\delta'|^2 =1$, $\alpha^{\prime} \delta^{\prime}  = \beta^{\prime} \gamma^{\prime}$, and $f(e) = \alpha' f(v_1) + \beta' w_2 + \gamma^{\prime} w_4 + \delta^{\prime} w_3$;
\item There exist a minimal tripotent $\tilde v\in C^{\prime}$, a rank two tripotent $u\in C^{\prime}$, and complex numbers $\alpha^{\prime}, \beta^{\prime}, \delta^{\prime}$ such that $(f(v_1), u,\tilde v)$ is a trangle, $|\alpha^{\prime}|^2 +2 | \beta^{\prime}|^2 + |\delta^{\prime}|^2 =1$, $\alpha^{\prime} \delta^{\prime}  = (\beta^{\prime})^2$, and $f(e) = \alpha^{\prime} f(v_1)+ \beta^{\prime} u +\delta^{\prime} \tilde v$.
\end{enumerate}

We shall first deal with case $(1)$. By hypothesis $\|v_1 \pm e \| =\|f(v_1) \pm f(e) \| .$ Applying Proposition \ref{p distance between minimal tripotents} we obtain: \begin{equation}\label{eq system 0} {(1 - \Re\hbox{e} \alpha)+ \sqrt{(1 -  \Re\hbox{e} \alpha)^2 -  |\delta|^2}}= \|v_1 - e \|^2
 \end{equation} $$= \|f(v_1) - f(e) \|^2= {(1 - \Re\hbox{e} \alpha')+ \sqrt{(1 -  \Re\hbox{e} \alpha')^2 -  |\delta'|^2}}$$ and $${(1 + \Re\hbox{e} \alpha)+ \sqrt{(1 +  \Re\hbox{e} \alpha)^2 -  |\delta|^2}}= \|v_1 + e \|^2 $$ $$= \|f(v_1) + f(e) \|^2= {(1 + \Re\hbox{e} \alpha')+ \sqrt{(1 +  \Re\hbox{e} \alpha')^2 -  |\delta'|^2}},$$ that is, \begin{equation}
 \label{eq system 1} (\Re\hbox{e} \alpha' - \Re\hbox{e} \alpha) + \sqrt{(1 -  \Re\hbox{e} \alpha)^2 -  |\delta|^2}= \sqrt{(1 -  \Re\hbox{e} \alpha')^2 -  |\delta'|^2},
 \end{equation} and \begin{equation}\label{eq system 2}{(\Re\hbox{e} \alpha- \Re\hbox{e} \alpha')+ \sqrt{(1 +  \Re\hbox{e} \alpha)^2 -  |\delta|^2}}= \sqrt{(1 + \Re\hbox{e} \alpha' )^2 -  |\delta'|^2}.\end{equation} It is not hard to check that the unique solution to the system formed by \eqref{eq system 1} and \eqref{eq system 2} is $$ |\delta'| = |\delta| =0, \hbox{ or } \Re\hbox{e} \alpha = \Re\hbox{e} \alpha' \hbox{ and } |\delta'| =  |\delta|.$$ In the case $|\delta'| = |\delta| =0,$ it follows from \eqref{eq system 0} that $\Re\hbox{e} \alpha = \Re\hbox{e} \alpha'$. We have therefore shown that $$\Re\hbox{e} \alpha= \Re\hbox{e} \alpha' \hbox{ and } |\delta'| =  |\delta|.$$

Now, Proposition \ref{p distance between minimal tripotents} and the hypothesis give $${(1 \pm \Im\hbox{m} \alpha)+ \sqrt{(1 \pm  \Im\hbox{m} \alpha)^2 -  |\delta|^2}}= \|v_1 \pm i e\|^2 = \|i (v_1 \pm i e)\|^2 =  \|i v_1 \mp  e\|^2 $$ $$=  \|f(i v_1) \mp  f(e)\|^2  = \|i f( v_1) \mp  f(e)\|^2 = \|- f( v_1) \mp i   f(e)\|^2 = \| f( v_1) \pm i   f(e)\|^2$$ $$= {(1 \pm \Im\hbox{m} \alpha')+ \sqrt{(1 \pm  \Im\hbox{m} \alpha')^2 -  |\delta'|^2}}.$$ Arguing as above, we get $\Im\hbox{m} \alpha' = \Im\hbox{m} \alpha$, and hence $\alpha = \alpha'$. We have therefore proved that $\alpha = \alpha',$ and $|\delta| = |\delta'|,$ and thus \begin{equation}\label{eq 1 1710} f(e) = \alpha f(v_1) + P_1(f(v_1)) (f(e)) + P_0(f(v_1)) (f(e)),
 \end{equation} with $\| P_0(f(v_1)) (f(e))\| = \|P_0( v_1 ) (e)\| =|\delta|$.\smallskip

We consider now case $(2)$. The same arguments given in case $(1)$ lead us to \eqref{eq 1 1710}.\smallskip

Since $f(i v_1) = i f(v_1)$, Proposition \ref{p quadrangles and trangles}$(b)$ gives $f(i v_j) = i f(v_j)$, for every $j\in\{2,3,4\}$. When in previous arguments we replace $v_1$ with $v_2$, $v_4$ and $v_3$ we obtain
\begin{equation}\label{eq 2 1710}f(e) = \beta f(v_2) + P_1(f(v_2)) (f(e)) + P_0(f(v_2)) (f(e)), \end{equation} 
\begin{equation}\label{eq 3 1710}f(e) = \delta f(v_3) + P_1(f(v_3)) (f(e)) + P_0(f(v_3)) (f(e)), \end{equation} 
and \begin{equation}\label{eq 4 1710}f(e) = \gamma f(v_4) + P_1(f(v_4)) (f(e)) + P_0(f(v_4)) (f(e)). \end{equation} 

Since $(f(v_1),f(v_2),f(v_3),f(v_4))$ is a quadrangle of minimal tripotents in $C^\prime$, and hence $\alpha f(v_1) + \beta f(v_2) + \gamma f(v_4) + \delta f(v_3)$ is a minimal tripotent in $C^\prime$, we deduce from \eqref{eq 1 1710}, \eqref{eq 2 1710}, \eqref{eq 3 1710} and \eqref{eq 4 1710} that $$P_2 (\alpha f(v_1) + \beta f(v_2) + \gamma f(v_4) + \delta f(v_3)) (f(e)) = \alpha f(v_1) + \beta f(v_2) + \gamma f(v_4) + \delta f(v_3),$$ and since $f(e)$ is a minimal tripotent in $C^\prime$, Lemma 1.6 in \cite{FriRu85} implies that $$f(e) = \alpha f(v_1) + \beta f(v_2) + \gamma f(v_4) + \delta f(v_3),$$ which concludes the proof of $(a)$.\smallskip

$(b)$ Let us assume that $f(i v) = i f(v)$  (the case $f(i v_1) = - i f(v_1)$ follows similarly). By Proposition \ref{p quadrangles and trangles}$(e)$ we have $f(i \tilde v) = i f(\tilde v)$ and $f(i u) = i f(u)$. Let $e$ denote $\alpha v + \beta u + \delta \tilde v$. As before, $f(e)$ is a minimal tripotent and $f(-e) =-f(e)$ (compare Theorem \ref{t Tingley antipodes for finite rank new}), and by Lemma \ref{l FR Prop 5 with a grid} one of the following statements holds:\begin{enumerate}[$(1)$]\item There exist minimal tripotents $w_2,w_3,w_4$ in $C^{\prime}$, and complex numbers $\alpha'$, $\beta'$, $\gamma'$, $\delta'$ such that $(f(v),w_2,w_3,w_4)$ is a quadrangle, $|\alpha'|^2 +| \beta'|^2 + |\gamma'|^2 + |\delta'|^2 =1$, $\alpha^{\prime} \delta^{\prime}  = \beta^{\prime} \gamma^{\prime}$, and $f(e) = \alpha' f(v) + \beta' w_2 + \gamma^{\prime} w_4 + \delta^{\prime} w_3$;
\item There exist a minimal tripotent $\tilde w\in C^{\prime}$, a rank two tripotent $u\in C^{\prime}$, and complex numbers $\alpha^{\prime}, \beta^{\prime}, \delta^{\prime}$ such that $(f(v), u,\tilde w)$ is a trangle, $|\alpha^{\prime}|^2 +2 | \beta^{\prime}|^2 + |\delta^{\prime}|^2 =1$, $\alpha^{\prime} \delta^{\prime}  = (\beta^{\prime})^2$, and $f(e) = \alpha^{\prime} f(v)+ \beta^{\prime} u +\delta^{\prime} \tilde w$.
\end{enumerate}

In case $(1)$, arguing as above we get $${(1 \mp \Re\hbox{e} \alpha)+ \sqrt{(1 \mp  \Re\hbox{e} \alpha)^2 -  |\delta|^2}}= \|v \mp e \|^2 $$ $$= \|f(v) \mp f(e) \|^2= {(1 \mp \Re\hbox{e} \alpha')+ \sqrt{(1 \mp  \Re\hbox{e} \alpha')^2 -  |\delta'|^2}},$$ from which we obtain $ \Re\hbox{e} \alpha' =  \Re\hbox{e} \alpha$.  Repeating previous arguments, we also have  $${(1 \pm \Im\hbox{m} \alpha)+ \sqrt{(1 \pm  \Im\hbox{m} \alpha)^2 -  |\delta|^2}}= \|v \pm i e\|^2 = \|i (v \pm i e)\|^2 =  \|i v \mp  e\|^2 $$ $$=  \|f(i v) \mp  f(e)\|^2  = \|i f( v) \mp  f(e)\|^2 = \|- f( v) \mp i   f(e)\|^2 = \| f( v) \pm i   f(e)\|^2$$ $$= {(1 \pm \Im\hbox{m} \alpha')+ \sqrt{(1 \pm  \Im\hbox{m} \alpha')^2 -  |\delta'|^2}},$$ and consequently $\Im\hbox{m} \alpha' = \Im\hbox{m} \alpha$, and $\alpha = \alpha'$. Therefore \begin{equation}\label{eq 1a 1710}  P_2(f(v)) (f(e)) = \alpha f(v).
 \end{equation} In case $(2)$ we also arrive to \eqref{eq 1a 1710} with similar arguments to those given above. This discussion remains valid when $v$ is replaced by $\tilde v$, and we therefore have
 \begin{equation}\label{eq 1b 1710} P_2(f(\tilde v)) (f(e)) =  \delta f(\tilde v) .
 \end{equation}

Now, we set $w = \frac12 (v+u+\tilde v)$, $\tilde w = \frac12 (v-u+\tilde v)$ and $\widehat{u} = v-\tilde v$. The triplet $(w,\widehat{u},\tilde w)$ is a trangle in $C$ and $e = \widehat{\alpha} w + \widehat{\beta} \widehat{u} + \widehat{\delta} \tilde w,$ where $\widehat{\alpha} = \frac{\alpha + 2 \beta + \delta}{2}$, $\widehat{\beta} = \frac{\alpha - \delta}{2}$, and $\widehat{\delta} = \frac{\alpha - 2 \beta + \delta}{2}.$ By the arguments given above we have \begin{equation}\label{eq 1a 2111}  P_2(f(w)) (f(e)) = \widehat{\alpha} f(w) = \frac{\alpha + 2 \beta + \delta}{2} f(w),
 \end{equation} and
\begin{equation}\label{eq 1b 2111}  P_2(f(\tilde w)) (f(e)) = \widehat{\delta} f(\tilde w) = \frac{\alpha - 2 \beta + \delta}{2} f(\tilde  w).
 \end{equation}

Having in mind \eqref{eq 1a 1710}, \eqref{eq 1b 1710}, \eqref{eq 1a 2111} and \eqref{eq 1b 2111}, and applying Lemma \ref{l new Spin} to the element $f(e)$ and the triplet $(f(v),f(u),f(\tilde v))$ (compare Proposition \ref{p quadrangles and trangles}$(d)$), we get
$$f(e) =  \alpha f(v) +\delta f(\tilde v) +\beta f(u) + P_0 ( \alpha f(v) +\delta f(\tilde v) +\beta f(u)) (f(e)),$$ and, by the minimality of $f(e)$, we deduce that $f(e) = \alpha f(v) +\delta f(\tilde v) +\beta f(u),$ as desired.
\end{proof}

\begin{corollary}\label{c linearity or anti-linearity for all minimal} Let $f: S(C) \to S(C^{\prime})$ be a surjective isometry between elementary JB$^*$-triples with rank$(C)\geq 2$. Then either  $f(i u) = i f(u)$ for every finite rank tripotent $u$ in $C$, or $f(i u) = - i f(u)$ for every finite rank tripotent $u$ in $C$.
\end{corollary}

\begin{proof}
Suppose there exists a minimal tripotent $e\in C$ such that $f(i e) = i f(e)$, and let $v$ be any other minimal tripotent in $C$. By Lemma \ref{l FR Prop 5 with a grid}  one of the following statements holds:\begin{enumerate}[$(a)$]\item There exist minimal tripotents $v_2,v_3,v_4$ in $C$, and complex numbers $\alpha$, $\beta$, $\gamma$, $\delta$ such that $(e,v_2,v_3,v_4)$ is a quadrangle, $|\alpha|^2 +| \beta|^2 + |\gamma|^2 + |\delta|^2 =1$, $\alpha \delta  = \beta \gamma$, and $v = \alpha e + \beta v_2 + \gamma v_4 + \delta v_3$, and hence $i v = i \alpha e + i \beta v_2 + i \gamma v_4 + i \delta v_3$;
\item There exist a minimal tripotent $\tilde e\in C$, a rank two tripotent $u\in C$, and complex numbers $\alpha, \beta, \delta$ such that $(e, u,\tilde e)$ is a trangle, $|\alpha|^2 +2 | \beta|^2 + |\delta|^2 =1$, $\alpha \delta  = \beta^2$, $v = \alpha e+ \beta u +\delta \tilde e$ and $i v = i \alpha e+i \beta u +i \delta \tilde e$.
\end{enumerate} Both cases will be treated independently. Proposition \ref{p quadrangles and trangles} assures that \begin{enumerate}[$(a)$]\item $f(i v_j)= i f(v_j)$, for every $j=2,3,4$;
\item $f(i u)= i f(u)$ and $f(i \tilde e)= i f(\tilde e)$.
\end{enumerate} An application of Theorem \ref{p additivity on trangles and quadrangles} proves \begin{enumerate}[$(a)$]\item $f(i v) =  i \alpha f(v_1) + i \beta f(v_2) + i \gamma f(v_4) +i  \delta f(v_3) = i f(v)$;
\item $f(iv ) = i {\alpha} f(e) + i {\beta} f(u) + i \delta f(\tilde e) = i f(v)$.
\end{enumerate}

The final statement is a consequence of the first conclusion and Theorem \ref{t Tingley antipodes for finite rank new}.
\end{proof}

Before finishing this section, we shall present another refinement of the Triple System Analyzer \cite[Proposition 2.1 $(iii)$]{DanFri87}.

\begin{lemma}\label{l preservation of the Peirce 1 subspace} Let $f: S(C) \to S(C^{\prime})$ be a surjective isometry between elementary JB$^*$-triples with rank$(C)\geq 2$, and let $e$ be a minimal tripotent in $C$. Then $f(S(C_1(e))) = S(C'_1(f(e)))$.
\end{lemma}

\begin{proof}
Let $e$ be a minimal tripotent in $C$, and let us pick a minimal tripotent $u$ in $C_1 (e)$. By the Triple System Analyzer (see \cite[Proposition 2.1 $(iii)$]{DanFri87}) either $u$ is minimal in $C$ and $e\top u$ or $u$ is not minimal in $C$, $u\vdash e$ and the triplet $(e, u, \tilde e =Q(u) (e))$ is a trangle with $e$ and $\tilde e$ minimal in $C$. Since in the first case, we can always find minimal tripotents $e_3$ and $e_4$ in $C$ such that $(e, u, e_3, e_4)$ is a quadrangle (compare the arguments in the proof of Proposition \ref{p quadrangles and trangles}$(c)$), we deduce, applying Proposition \ref{p quadrangles and trangles}, that $f(u)\in C_1'(f(e)).$\smallskip

Since $C_1(e)$ has rank one or two (see \cite[Corollary 2.2]{DanFri87}), given an element $x$ in the unit sphere of $C_1(e)$, one of the following holds\begin{enumerate}[$(a)$]\item $x$ is a minimal tripotent in $C$ (this happens when $C_1(e)$ has rank one);
\item We can find two orthogonal minimal tripotents $u_1,u_2\in C_1(e)$ and $\lambda\in \mathbb{R}$ such that $x = u_1  + \lambda u_2$ and $|\lambda |\leq 1$ (compare \cite[Remark 4.6]{BuChu}).
\end{enumerate} In case $(a)$,  by the arguments the first paragraph, we have $f(x)\in C_1'(f(e))$. In case $(b)$ we observe that, by the Triple System Analyzer, $u_1$ and $u_2$ are finite rank tripotents in $C$, and thus, Proposition \ref{p surjective isometries between the spheres preserve finite rank tripotents} it follows that $$f(x) =  f(u_1  + \lambda u_2) = f( u_1)  + T_{u_1} (\lambda u_2) =  f(u_1)  + \lambda f(u_2) \in  C_1'(f(e)).$$
\end{proof}

\section{Synthesis of a real linear isometry}\label{sec: sintesis}

In a tour the force, T. Dang and Y. Friedman \cite{DanFri87} and E. Neher \cite{Neher87} developed, independently, a complete theory on coordinatization theorems for the Jordan triple systems ``covered'' by a ``grid''. A \emph{grid} in a JB$^*$-triple $E$ is a family of minimal and rank two tripotents in $E$ built up of quadrangles of minimal tripotents or trangles of the form $(v,u, \tilde v)$ with $v$ and $\tilde v$ minimal, where  all the non-vanishing triple products among the elements of the grid are those associated to these types of trangles and quadrangles. A typical grid in the Cartan factor $M_{n,m} (\mathbb{C})$ is given by the family of all matrix units.\smallskip

The results in \cite{DanFri87} and \cite{Neher87} prove, among other classification theorems, that every Cartan factor $C$ admits a (rectangular, symplectic, hermitian, spin, or exceptional) grid $\mathcal{G}$ such that the elementary JB$^*$-triple $K$ associated with $C$ is precisely the norm closed linear span of the grid $\mathcal{G}$, and $C$ being the weak$^*$-closure of $K$ is nothing but the weak$^*$-closure of the linear span of $\mathcal{G}$ (compare \cite[Structure Theorem IV.3.14]{Neher87} or \cite[\S 2]{DanFri87}). A more detailed description of the grids will be given in subsequent results.\smallskip

We can now state the main result of this section.

\begin{theorem}\label{thm Tingley thm ofr weakly compact JB*-triples} Let $f: S(C) \to S(C^{\prime})$ be a surjective isometry between the unit spheres of two elementary JB$^*$-triples with rank greater or equal than two. Then there exists a surjective complex linear or conjugate linear isometry $T: C\to C^{\prime}$ satisfying $T|_{S(C)} = f$.
\end{theorem}

The proof will follow from Theorems \ref{thm Tingley thm for type 1 rank 2, 3 and 4}, \ref{thm Tingley thm for type 2 rank 2, 3 and 4}, \ref{thm Tingley thm for type 3 rank 2, 3 and 4}, \ref{thm Tingley thm for finite dimensional}, and \ref{thm Tingley thm for spin factors} below. These results will be obtained by an individualized approach on each elementary JB$^*$-triple.\smallskip

\begin{remark}\label{remark conjugate}{\rm Let $f: S(C) \to S(C^{\prime})$ be a surjective isometry between the unit spheres of two elementary JB$^*$-triples with rank greater or equal than two. By Corollary \ref{c linearity or anti-linearity for all minimal} we know that $f(i e) = i f(e)$ or $f(i e) = - i f(e)$, for every finite rank tripotent $e\in C$. In the second case, we can always replace $C'$ with the complex JB$^*$-triple $C''$ obtained from $C'$ by keeping the original norm, triple product, and sum of vectors but replacing the product by scalars with the product given by $\lambda \cdot x = \overline{\lambda} x$ ($\lambda\in \mathbb{C}$, $x\in C'$). Then the mapping $\overline{f} : S(C) \to S(C^{\prime\prime}),$ $x\mapsto f(x)$ is a surjective isometry with $\overline{f} (i e) = i \overline{f}(e)$ for every minimal tripotent $e\in C$. If there exists a surjective complex linear isometry $\overline{T} : C \to C''$ extending the mapping $\overline{f}$, then we can easily find a conjugate linear isometry  ${T} : C \to C'$, $T(x) =\overline{T}(x)$, whose restriction to $S(C)$ is precisely $f$.}
\end{remark}

\begin{remark}\label{r complex linearity or anti-linearity}{\rm Suppose $f: S(C) \to S(C^{\prime})$ is a surjective isometry between the unit spheres of two elementary JB$^*$-triples with rank$(C)\geq 2$. Let $e$ be a minimal tripotent in $C$ and let $T_{e} : C_0(e)\to C'_{0} (f(e))$ be the surjective real linear isometry given by Proposition \ref{p surjective isometries between the spheres preserve finite rank tripotents}$(c)$. If rank$(C_0(e))\geq 2$, it follows from \cite[Proposition 2.6]{Da} that $T_e$ either is complex linear if $f(i e) = i f(e)$ or conjugate linear if $f(i e) = - i f(e)$ (compare Corollary \ref{c linearity or anti-linearity for all minimal}). When $C_0(e)$ has rank $1$ (and hence it is a complex Hilbert space regarded as a type 1 Cartan factor), every element in $S(C_0(e))$ is a minimal tripotent in $C_0(e)$. Therefore, it follows from Corollary \ref{c linearity or anti-linearity for all minimal} that $T_e$ either is complex linear if $f(i e) = i f(e)$ or conjugate linear if $f(i e) = - i f(e)$.\smallskip

Actually, if $F$ is a JB$^*$-subtriple of $C$, having in mind that a JB$^*$-triple of a weaklyc ompact JB$^*$-triple is weakly compact, every element in $F$ can be approximated in norm by a finite linear combination of mutually orthogonal minimal tripotents in $F$ (see \cite[Remark 4.6]{BuChu}). Moreover, every minimal tripotent in $F$ is a finite rank tripotent in $C$. Therefore, given a bounded real linear operator $T_1 :  F\to C'$ such that $T_1 (x) = f(x)$ for every $x\in S(F)$, we deduce from Corollary \ref{c linearity or anti-linearity for all minimal} that $T_1$ either is complex linear if $f(i e) = i f(e)$ or conjugate linear if $f(i e) = - i f(e)$, for any minimal tripotent $e\in C$.}
\end{remark}

\subsection{Elementary JB$^*$-triples of type 1} We begin our particular study for an elementary JB$^*$-triple $C$ of type 1 and rank between 2 and 4. We are mainly interested in the case $C=L(H,H^\prime),$ where $H$ and $H^\prime$ are complex Hilbert spaces with $2\leq\min\{$dim$(H'),\hbox{dim} (H)\}\leq 4$ (see Section \ref{sec:2}), however the next result is established under more general hypothesis.

\begin{theorem}\label{thm Tingley thm for type 1 rank 2, 3 and 4} Let $C=K(H,H^\prime),$ where $H$ and $H^\prime$ are complex Hilbert spaces with $2\leq\min\{$dim$(H')$, dim$(H)\}$, and let $C'$ be an elementary JB$^*$-triple. Suppose $f: S(C) \to S(C^{\prime})$ is a surjective isometry. Then there exists a surjective complex linear or conjugate linear isometry $T: C\to C^{\prime}$ satisfying $T|_{S(C)} = f$.
\end{theorem}

\begin{proof} Let us first assume that $f(i e) = i f(e)$, for every minimal tripotent $e\in C$ (compare Corollary \ref{c linearity or anti-linearity for all minimal}). We deduce from Remark \ref{r complex linearity or anti-linearity} that the operator $T_{e}$ given by Proposition \ref{p surjective isometries between the spheres preserve finite rank tripotents}$(c)$ is complex linear.\smallskip

Let $\{\xi_i : i \in I\}$ and $\{\eta_j :  j \in J\}$ be orthonormal basis of $H$ and $H'$, respectively. We set $u_{ij} := \eta_j\otimes \xi_i,$ $(i,j)\in I\times J$. Then, the family $\{u_{ij} : (i,j)\in I\times J\}$ is a \emph{rectangular grid} in $C$ (compare \cite{DanFri87}, \cite{Neher87}), however we will not make an explicit use of the properties of the grid in this case.\smallskip

To simplify the notation, we assume that $1,2\in I, J$. Let us consider the minimal tripotents $u_{11}, u_{12}, u_{21}, u_{22}$, and for each one of them the surjective real linear isometry $T_{u_{ij}}: C_0(u_{ij})\to C'_0 (f(u_{ij})).$\smallskip

We can decompose $C$ in the form $$C = \mathbb{C} u_{11} \oplus \left(C_0 (u_{21})\cap C_1(u_{11})\right) \oplus \left(C_0 (u_{12})\cap C_1(u_{11})\right) \oplus C_0 (u_{11}).$$ Let $P_{10} = P_1 (u_{11}) P_0 (u_{21}) = P_0 (u_{21})  P_1 (u_{11})$, $P_{01}=P_1 (u_{11}) P_0 (u_{12})$. The uniqueness of the above decomposition shows that the mapping $T : C \to C'$ given by $$T(x) = P_2(f(u_{11})) (x) + T_{u_{21}} (P_{10} (x)) + T_{u_{12}} (P_{01} (x)) + T_{u_{11}} (P_0 (u_{11}) (x))$$ is a well defined bounded real linear operator.\smallskip

Let $u=\eta\otimes \xi$ be a minimal tripotent in $C$ with $\|\eta\|=1=\|\xi\|.$ A concrete decomposition similar to that given by Lemma \ref{l FR Prop 5 with a grid} can be materialized as follows: let us write $\eta =\lambda_1 \eta_1 +\lambda_2 \eta_0$ and $\xi =\mu_1 \xi_1 + \mu_2 \xi_0,$ where $\|\eta_0\|=1=\|\xi_0\|,$ $\eta_1\perp \eta_0$, $\xi_0\perp \xi_1$ (in the Hilbertian sense), $\lambda_1, \lambda_2, \mu_1 , \mu_2 \in \mathbb{C}$ with $|\lambda_1|^2+ |\lambda_2|^2=1$, and $|\mu_1|^2 +| \mu_2 |^2=1$. Thus, we have $$ u = \lambda_1 \overline{\mu}_1 u_{11}+
\lambda_1 \overline{\mu}_2 w_{12}+ \lambda_2 \overline{\mu}_1  w_{21}+ \lambda_2  \overline{\mu}_2 w_{0}, $$ where $w_{12} = \eta_1 \otimes \xi_0,$ $w_{21} = \eta_0 \otimes \xi_1$, $w_{0} =\eta_0 \otimes \xi_0$, where $(u_{11}, w_{12}, w_0, w_{21})$ is a quadrangle of minimal tripotents, $w_{12}\in C_1(u_{11})\cap C_0(u_{21})$, $w_{21}\in C_1(u_{11})\cap C_0(u_{12})$, $w_{0}\in C_0(u_{11})$. We are in position to apply Theorem \ref{p additivity on trangles and quadrangles}$(a)$ and the complex linearity of $T(u_{11})$, $T(u_{12}),$ and $T(u_{21})$ to deduce that $$f(u) = \lambda_1 \overline{\mu}_1 f(u_{11})+
\lambda_1 \overline{\mu}_2 f(w_{12})+ \lambda_2 \overline{\mu}_1  f(w_{21})+ \lambda_2  \overline{\mu}_2 f(w_{0}) $$ $$= \lambda_1 \overline{\mu}_1 T(u_{11})+
\lambda_1 \overline{\mu}_2 T_{u_{21}} (w_{12})+ \lambda_2 \overline{\mu}_1  T_{u_{12}} (w_{21})+ \lambda_2  \overline{\mu}_2 T_{u_{11}} (w_{0}) $$ $$=  T(\lambda_1 \overline{\mu}_1 u_{11})+
 T_{u_{21}} (\lambda_1 \overline{\mu}_2 w_{12})+  T_{u_{12}} ( \lambda_2 \overline{\mu}_1 w_{21})+ T_{u_{11}} ( \lambda_2  \overline{\mu}_2 w_{0}) $$ $$=  T(\lambda_1 \overline{\mu}_1 u_{11})+
 T (\lambda_1 \overline{\mu}_2 w_{12})+  T ( \lambda_2 \overline{\mu}_1 w_{21})+ T ( \lambda_2  \overline{\mu}_2 w_{0}) = T(u).$$ We observe that $T$ is actually complex linear.\smallskip

We have therefore shown that $f(u) = T(u)$, for every minimal tripotent $u\in C$. Proposition 3.9 in \cite{PeTan16} concludes that $T|_{S(C)} = f$.\smallskip

Finally, if $f(i e) = - i f(e)$, for every minimal tripotent $e\in C$. Let $\overline{f}$, and $C''$ be the mapping and the elementary JB$^*$-triple defined in Remark \ref{remark conjugate}. The arguments above show that we can find a surjective real linear isometry $\overline{T} : C\to C''$ such that $\overline{T}|_{S(C)} = \overline{f}$. The arguments in the just quoted Remark show the existence of a surjective conjugate linear isometry ${T} : C\to C'$ satisfying $T|_{S(C)} = f$.
\end{proof}

\subsection{Elementary JB$^*$-triples of types 2 and 3} In the next results we deal with elementary JB$^*$-triples of type 2 and 3 with rank greater or equal than 2. For this reason we fix a complex Hilbert space $H$, a conjugation $j : H\to H$, and the complex linear involution on $L(H)$ defined by $x^t = j x^* j$ ($x\in L(H)$).

\begin{theorem}\label{thm Tingley thm for type 2 rank 2, 3 and 4} Let $C=\{x\in K(H): x^t = -x\}$ with rank$(C)\geq 2$ (i.e. dim$(H)\geq 4$), and let $C'$ be an elementary JB$^*$-triple. Suppose $f: S(C) \to S(C^{\prime})$ is a surjective isometry. Then there exists a surjective complex linear or conjugate linear isometry $T: C\to C^{\prime}$ satisfying $T|_{S(C)} = f$.
\end{theorem}

\begin{proof} As in the proof of the previous theorem, we first assume that $f(i e) = i f(e)$, for every minimal tripotent $e\in C$ (compare Corollary \ref{c linearity or anti-linearity for all minimal}). The case $f(i e) = - i f(e)$, for every minimal tripotent $e\in C$ follows by similar arguments.\smallskip

Let $\{\xi_i : i \in I\}$ be an orthonormal basis of $H$. Defining $u_{ij} = j(\xi_i)\otimes \xi_j - j(\xi_j)\otimes \xi_i$ ($i,j\in I$), the set $\{u_{ij}: i\neq j \hbox{ in } I\}$ is a \emph{sympletic grid} in the sense of \cite[page 317]{DanFri87}.\smallskip

The element $u_{12}$ is a minimal tripotent in $C$ and $f(u_{12})$ satisfies the same property (see Proposition \ref{p surjective isometries between the spheres preserve finite rank tripotents}). By Proposition \ref{p surjective isometries between the spheres preserve finite rank tripotents}$(c)$, there exists a surjective real linear isometry $T_{u_{12}} : C_0 (u_{12}) \to  C'_0 (f(u_{12}))$ satisfying $$f(u_{12} +x )= f(u_{12}) + T_{u_{12}} (x),$$ for every $x\in \mathcal{B}_{C_0(u_{12})}$, and $T_{u_{12}} (x) = f(x)$ for all $x\in S(C_0(u_{12}))$.\smallskip

By Lemma \ref{l preservation of the Peirce 1 subspace}, $f|_{S(C_1(u_{12}))} : S(C_1(u_{12})) \to S(C'_1(f(u_{12})))$ is a surjective isometry. Let $p = \xi_1 \otimes \xi_1 + \xi_2\otimes \xi_2$, $\tilde p = j(\xi_1) \otimes j(\xi_1) + j(\xi_2)\otimes j(\xi_2)$. Then $p$ and $\tilde p$ are rank-2 projections in $B(H)$. The space $C_1(u_{12})$ is isometrically isomorphic to $B(p(H), (1-\tilde p) (H))$ via the mapping $a\mapsto (1-\tilde p) a p$. Applying Theorem \ref{thm Tingley thm for type 1 rank 2, 3 and 4} we deduce the existence of a surjective real linear isometry $T_1 : C_1(u_{12}) \to C'_1(f(u_{12}))$ satisfying $T_1 (x) = f(x)$ for every $x\in {S(C_1(u_{12}))}$.\smallskip

The uniqueness of the Peirce decomposition $C = \mathbb{C} u_{12}\oplus C_1(u_{12}) \oplus C_1(u_{12})$ and the real linearity of the mappings $T_{u_{12}}$ and $T_1$ guarantee that the mapping $T: C\to C'$, defined by $$T(x) = T(\lambda u_{12} +P_1 (u_{12}) (x)+ P_0 (u_{12}) (x)) $$ $$= \lambda f(u_{12}) + T_1 (P_1 (u_{12}) (x)) + T_{u_{12}} (P_1 (u_{12}) (x))$$ is a well-defined continuous linear operator with $\|T\|\leq 3$.\smallskip

Let $e$ be any minimal tripotent in $C$. By Lemma \ref{l FR Prop 5 with a grid} one of the following statements holds:\begin{enumerate}[$(a)$]\item There exist minimal tripotents $v_2,v_3,v_4$ in $C$, and complex numbers $\alpha$, $\beta$, $\gamma$, $\delta$ such that $(u_{12},v_2,v_3,v_4)$ is a quadrangle, $|\alpha|^2 +| \beta|^2 + |\gamma|^2 + |\delta|^2 =1$, $\alpha \delta  = \beta \gamma$, and $e = \alpha u_{12} + \beta v_2 + \gamma v_4 + \delta v_3$;
\item There exist a minimal tripotent $\tilde u_{12}\in C$, a rank two tripotent $u\in C$, and complex numbers $\alpha, \beta, \delta$ such that $(u_{12}, u,\tilde u_{12})$ is a trangle, $|\alpha|^2 +2 | \beta|^2 + |\delta|^2 =1$, $\alpha \delta  = \beta^2$, and $e = \alpha u_{12}+ \beta u +\delta \tilde u_{12}$.
\end{enumerate} Applying Theorem \ref{p additivity on trangles and quadrangles} and the definition of $T$ we get
\begin{enumerate}[$(a)$]\item $$f(e) = f(\alpha u_{12} + \beta v_2 + \gamma v_4 + \delta v_3) = \alpha f(u_{12}) + \beta f(v_2) + \gamma f(v_4) + \delta f(v_3)$$ $$= \alpha f(u_{12}) + \beta T_1(v_2) + \gamma T_1(v_4) + \delta T_{u_{12}} (v_3)$$ $$=\hbox{(by Remark \ref{r complex linearity or anti-linearity})}= f( \alpha u_{12}) + T_1( \beta v_2) + T_1( \gamma v_4) +  T_{u_{12}} (\delta v_3) $$ $$= T( \alpha u_{12}) + T( \beta v_2) + T( \gamma v_4) + T( \delta v_3) = T(e),$$ or,
\item $$f(e) = f(\alpha u_{12}+ \beta u +\delta \tilde u_{12})= \alpha f(u_{12})+ \beta f(u) +\delta f(\tilde u_{12})$$ $$= \alpha T(u_{12})+ \beta T_1(u) +\delta T_{u_{12}} (\tilde u_{12}) $$ $$=\hbox{(by Remark \ref{r complex linearity or anti-linearity})} = T( \alpha u_{12})+ T_1( \beta u) +T_{u_{12}} (\delta \tilde u_{12}) = T(e),$$
\end{enumerate} respectively.\smallskip

We have therefore shown that $f(e) = T(e)$ for every minimal tripotent $e\in C$. Proposition 3.9 in \cite{PeTan16} proves that $T(x) = f(x)$ for every $x\in S(C)$, and hence $T$ is surjective and isometric. We note that $T$ actually is complex linear.
\end{proof}

We shall deal next with a Cartan factor of type 3.

\begin{theorem}\label{thm Tingley thm for type 3 rank 2, 3 and 4} Let $C=\{x\in K(H): x^t = x\},$ where dim$(H)\geq 2$, and let $C'$ be an elementary JB$^*$-triple. Suppose $f: S(C) \to S(C^{\prime})$ is a surjective isometry. Then there exists a surjective complex linear or conjugate linear isometry $T: C\to C^{\prime}$ satisfying $T|_{S(C)} = f$.
\end{theorem}

\begin{proof} The proof follows similar guidelines to the proof of Theorem \ref{thm Tingley thm for type 2 rank 2, 3 and 4}. We may assume thanks to Corollary \ref{c linearity or anti-linearity for all minimal}, that $f(i e) = i f(e)$, for every minimal tripotent $e\in C$. The case $f(i e) = - i f(e)$, for every minimal tripotent $e\in C$ is very similar.\smallskip

Let $\{\xi_i : i \in I\}$ be an orthonormal basis of $H$. Defining $u_{ij} = (j(\xi_i)\otimes \xi_j + j(\xi_j)\otimes \xi_i)$ ($i\neq j\in I$), and $u_{ii} = (j(\xi_i)\otimes \xi_i + j(\xi_i)\otimes \xi_i)$ ($i\in I$), the set $\{u_{ij}: i,j \in I\}$ is a \emph{hermitian grid} in the sense of \cite[page 308]{DanFri87}.\smallskip

Let $T_{u_{11}} : C_0 (u_{11}) \to  C'_0 (f(u_{11}))$ be the surjective real linear isometry satisfying $$f(u_{11} +x )= f(u_{11}) + T_{u_{11}} (x),$$ for every $x\in \mathcal{B}_{C_0(u_{11})}$, and $T_{u_{11}} (x) = f(x)$ for all $x\in S(C_0(u_{11}))$, whose existence is guaranteed by Proposition \ref{p surjective isometries between the spheres preserve finite rank tripotents}$(c)$.\smallskip

Lemma \ref{l preservation of the Peirce 1 subspace} implies that $f|_{S(C_1(u_{11}))} : S(C_1(u_{11})) \to S(C'_1(f(u_{11})))$ is a surjective isometry. The elementary JB$^*$-triple $C_1(u_{11})$ has rank one, and hence, it is isometrically isomorphic to a Hilbert space. It follows from \cite[Theorem 2.1]{Ding2002} (see also \cite[Corollary 3.15]{PeTan16}) that there exists a surjective real linear isometry $T_1 : C_1(u_{11}) \to C'_1(f(u_{11}))$ satisfying $T_1 (x) = f(x)$ for every $x\in {S(C_1(u_{11}))}$.\smallskip

Remark \ref{r complex linearity or anti-linearity} guarantees that $T_1$ and $ T_{u_{12}}$ are complex linear because we have assumed that $f(i e) = i f(e)$, for every minimal tripotent $e\in C$.\smallskip

Repeating the arguments in the final part of the proof of Theorem \ref{thm Tingley thm for type 2 rank 2, 3 and 4} we deduce that the mapping $T: C\to C'$, defined by $$T(x) = T(\lambda u_{12} +P_1 (u_{12}) (x)+ P_0 (u_{12}) (x)) $$ $$= \lambda f(u_{12}) + T_1 (P_1 (u_{12}) (x)) + T_{u_{12}} (P_1 (u_{12}) (x)),$$ is a well-defined continuous linear operator, and $T(x) = f(x)$ for every $x\in C$, which concludes the proof.
\end{proof}

\subsection{Finite dimensional elementary JB$^*$-triples} Surjective isometries between finite-dimensional Cartan factors can be treated by a unified approach. We recall that Cartan factors of types 5 and 6 are finite dimensional.

\begin{theorem}\label{thm Tingley thm for finite dimensional} Let $C, C'$ be elementary JB$^*$-triples with dim$(C)<\infty$. Suppose $f: S(C) \to S(C^{\prime})$ is a surjective isometry. Then there exists a surjective real isometry $T: C\to C^{\prime}$ satisfying $T|_{S(C)} = f$. Furthermore, if rank$(C)\geq 2$ then $T$ is complex linear or conjugate linear.
\end{theorem}

\begin{proof} We shall argue by induction on the dimension of $C$. If rank$(C)=1$ the conclusion follows from \cite[Theorem 2.1]{Ding2002} (see also \cite[Corollary 3.15]{PeTan16}). If the dim$(C) =1$ the statement has been proved.\smallskip

Henceforth, we assume rank$(C)\geq 2$.\smallskip

Let us assume that our statement is true for any Cartan factor $\tilde C$ with dimension $\leq n$, and dim$(C) = n+1$. Let us pick a minimal tripotent $e\in C$. We can assume, via Corollary \ref{c linearity or anti-linearity for all minimal}, that $f(i e) = i f(e)$ (the case $f(i e) = -i f(e)$ follows with similar techniques).\smallskip

By Lemma \ref{l preservation of the Peirce 1 subspace} $f|_{S(C_1(e))}: S(C_1(e))\to  S(C'_1(f(e)))$ is a surjective isometry. Since dim$(C_1(e))\leq n$, and $C_1(e)$ being generated by a standard grid (see \cite[\S IV.3]{Neher87}) is another Cartan factor, we conclude from the induction hypothesis that there exists a surjective real linear isometry $T_1 : C_1(e) \to C_1(e)$ satisfying $T_1 (x) = f(x)$ for all $x\in S(C_1(e))$. By Proposition \ref{p surjective isometries between the spheres preserve finite rank tripotents}$(c)$ we can find a surjective real linear isometry $T_e : C_0(e) \to C_0(e)$ such satisfying $f(x) = T_{e} (x)$ for all $x\in S(C_0(e))$.\smallskip

We define a bounded linear mapping $T: C \to C'$ given by $$T(x) = T(\lambda e + P_1 (e) (x) +P_0(e)(x)) :=  \lambda f(e) + T_1 (P_1 (e) (x)) +  T_e(P_0(e)(x)).$$ The arguments given in the last three paragraphs of the proof of Theorem \ref{thm Tingley thm for type 2 rank 2, 3 and 4} can be now repeated to show that $T(u) = f(u)$ for every minimal tripotent $u\in C$. By \cite[Proposition 3.9]{PeTan16} we have $f(x) = T(x)$ for every $x\in S(C)$, which concludes the induction argument.\smallskip

The final observation follows as a consequence of Remark \ref{r complex linearity or anti-linearity}.
\end{proof}

\subsection{Spin factors} In this section we explore Tingley's problem for surjective isometries from the unit sphere of a spin factor into the unit sphere of an elementary JB$^*$-triple.\smallskip

The starting lemma shows that we cannot find a surjective isometry from the unit sphere of an infinite dimensional spin factor onto the unit sphere of a type 1 Cartan factor.

\begin{lemma}\label{l no surjective isometries between the s of spin and M2n}  Let $X$ be a spin factor, and let $C=L(H,H^\prime)$, where $H$ and $H^\prime$ are complex Hilbert spaces with dim$(H^{\prime})=2$, and dim$(H)\geq 3$. Then there is no surjective isometry $f: S(X) \to S(C)$.
\end{lemma}

\begin{proof} Suppose we can find a surjective isometry $g: S(C) \to S(X)$. By hypothesis, we can find at least three minimal tripotents $e_{11}, e_{12}$ and $e_{3}$ in $C$ satisfying $e_{11} \top e_{12},$ $e_{11}$ and $e_{12}$ generate a 2-dimensional complex Hilbert space, and $e_{11}, e_{12}\perp e_{3}$ (take for example $e_{11} = \eta_1\otimes \xi_1$, $e_{12} = \eta_1\otimes \xi_2$, and $e_{3} = \eta_2\otimes \xi_3$, where $\{\eta_1,\eta_2\}$ and $\{\xi_1,\xi_2, \xi_3\}$ are orthonormal systems in $H^\prime$ and $H$, respectively). Let us consider the surjective real linear isometry $T_{e_3} : C_0(e_3) \to X_0(f(e_3))$ given by Proposition \ref{p surjective isometries between the spheres preserve finite rank tripotents}$(c)$. The elements $e_{11}, e_{12},$ and $\frac{1}{\sqrt{2}} (e_{11}+e_{12})$ are minimal tripotents in $C_0(e_3)$.\smallskip

Since $f(e_3)$ is a minimal tripotent in a spin factor $X$, its orthogonal complement $X_0( f(e_3))=\mathbb{C} \overline{f(e_3)}$ is a one dimensional complex space, where $\overline{\ \cdot \ }$ denotes the conjugation on $X$ (compare \eqref{eq Peirce zero in spin is one dimensional}). The minimal tripotents $f(e_{11}), f(e_{12}),$ and $f(\frac{1}{\sqrt{2}} (e_{11}\pm e_{12}))$ belong to $X_0( f(e_3))$. So, there exist $\lambda_1$ and $\lambda_2$ in the unit sphere of $\mathbb{C}$ such that $f(e_{11}) =\lambda_1 \overline{f(e_3)}$ and $f(e_{12}) =\lambda_2 \overline{f(e_3)}$. The elements $$ f\Big(\frac{1}{\sqrt{2}} (e_{11}\pm e_{12})\Big) = T_{e_3} \Big(\frac{1}{\sqrt{2}} (e_{11}\pm e_{12})) \Big) = \frac{1}{\sqrt{2}} \Big(T_{e_3}(e_{11})\pm T_{e_3}(e_{12})\Big)$$ $$= \frac{1}{\sqrt{2}} \Big(f(e_{11})\pm f(e_{12})\Big)= \frac{1}{\sqrt{2}} (\lambda_1\pm \lambda_2)  \overline{f(e_3)}$$ must be norm-one, which implies $\lambda_2 = (-1)^k i \lambda_1$ for some natural $k$. We observe that, by Remark \ref{r complex linearity or anti-linearity}, $T_{e_3}$ is complex linear or conjugate linear. Arguing as above, we prove the existence of a natural $m$ such that $$ f\Big(\frac{1}{\sqrt{2}} (e_{11}\pm i e_{12})\Big) = T_{e_3} \Big(\frac{1}{\sqrt{2}} (e_{11}\pm i e_{12})) \Big) = \frac{1}{\sqrt{2}} \Big(T_{e_3}(e_{11})\pm (-1)^m i   T_{e_3}(e_{12})\Big) $$ $$= \frac{1}{\sqrt{2}} \Big(f(e_{11})\pm  (-1)^m i f( e_{12})\Big) = \frac{1}{\sqrt{2}} (\lambda_1\pm (-1)^{k+1+m} \lambda_1)  \overline{f(e_3)}$$ $$= \frac{1\pm (-1)^{k+1+m}}{\sqrt{2}} \lambda_1 \overline{f(e_3)}$$ must be norm-one too, which is impossible.
\end{proof}

\begin{theorem}\label{thm Tingley thm for spin factors} Let $X$ be a spin factor, let $C'$ be an elementary JB$^*$-triple, and let $f: S(X) \to S(C')$ be a surjective isometry. Then there exists a surjective real linear isometry $T: X\to C'$ satisfying $T|_{S(X)} = f$. Furthermore, the operator $T$ can be chosen to be complex linear or conjugate linear.
\end{theorem}

\begin{proof} If $X$ is finite dimensional the conclusion follows from Theorem \ref{thm Tingley thm for finite dimensional}. We can therefore assume that $X$ is infinite dimensional. Theorem \ref{t Tingley antipodes for finite rank new} implies that $X$ and $C'$ are infinite dimensional rank 2 Cartan factors. Therefore $C'$ is either a spin factor or a type 1 Cartan factor of the form $B(H,H')$, where $H$, $H'$ are complex Hilbert spaces, dim$(H')=2$, and dim$(H)=\infty$. Lemma \ref{l no surjective isometries between the s of spin and M2n} shows that $C'=Y$ must be an infinite dimensional spin factor.\smallskip

As in previous results, we can assume that $f(i e) = i f(e)$ for every minimal tripotent $e$ in $X$ (compare Corollary \ref{c linearity or anti-linearity for all minimal} and Remark \ref{r complex linearity or anti-linearity}). Applying that $X$ has rank two and Theorem \ref{t Tingley antipodes for finite rank new} we deduce that $f(i u) = i f(u)$ for every tripotent $u$ in $X$.\smallskip

By a little abuse of notation, the involutions on $X$ and on $Y$ will be both denoted by the symbol $\overline{\ \cdot \ }$, similarly, we shall indistinctly write $(.|.),$ $\|.\|_2$, and $\|.\|$ for the inner products, the Hilbertian norms, and the spin norms on $X$ and $Y$. The symbols $X_1$, $X_2$, $Y_1$ and $Y_2$ will have the usual meanings commented above, that is, $X_1 =\{x\in X : x= \overline{x} \}$, $Y_1 =\{y\in Y : y= \overline{x} \}$, $X_2 = i X_1$, and $Y_2 = i Y_1$.\smallskip

By the arguments in the first paragraph, we can assume that $X$ (and hence $X_1$) is infinite dimensional.\smallskip

In a first step we shall show that given $x_1,x_2$ in $S(X_1)$ with $(x_1|x_2) =0$ we have $$f(x_1) = \mu_1 y_1, f(x_2) = \mu_2 y_2,$$ for suitable $y_1,y_2\in S(Y_1)$, $\mu_1,\mu_2\in \mathbb{C}$, with $|\mu_j|=1$, $\mu_2 \in \{\pm \mu_1\}$, and $(y_1|y_2)=0$. Indeed, it is known that, for each $j=1,2$, $x_j$ is a rank two tripotent which is also unitary in $X$ (see the remarks before Lemma \ref{l no surjective isometries between the s of spin and M2n}). By Theorem \ref{t Tingley antipodes for finite rank new} $f(x_j)\in \hbox{max } \mathcal{U} (Y)$, and hence there exists a (unique) norm one element $y_j$ in $Y_1$ and $\mu_j\in \mathbb{C}$ with $|\mu_j|=1$ such that $f(x_j) = \mu_j y_j$.\smallskip

The elements $e = \frac{x_1+i x_2}{2}$, $\overline{e} = \frac{x_1-i x_2}{2}$ are minimal tripotents in $X$ with $e\perp \overline{e}$ (compare the comments before Lemma \ref{l no surjective isometries between the s of spin and M2n}). Since $f(e)$ and $f(\overline{e})$ are orthogonal minimal tripotents in $Y$, $\mu_1 y_1 = f(x_1) = f(e) + f(\overline{e})$ and, since we have assumed that $f(i u) = i f(u)$ for every tripotent $u$ in $X$, we have $i \mu_2 y_2 = f(i x_2) = f(e_1) - f(\overline{e}_1)$ (see Theorem \ref{t Tingley antipodes for finite rank new}). Therefore $$ f(e) = \frac{ \mu_1 y_1 +i \mu_2 y_2}{2}, \hbox{ and } f(\overline{e}) = \frac{ \mu_1 y_1 -i \mu_2 y_2}{2}.$$ Having in mind that $f(\overline{e})\perp f(e)$ in $Y$, the identity $$f(e) =\{f(e),f(e),f(e)\} = 2 (f(e)|f(e)) f(e) - (f(e)|\overline{f(e)}) \overline{f(e)}$$ proves that $ \frac12 = (f(e)|f(e)) $ and \begin{equation}\label{eq 2 1011} 0 = (f(e)|\overline{f(e)}) = \frac14 (\mu_1^2-\mu_2^2 + i \mu_1 {\mu_2} (y_1 | y_2) +i \mu_2 {\mu_1} (y_2|y_1)).
 \end{equation}
Replacing $f(e)$ with $f(\overline{e})$ we also get \begin{equation}\label{eq 3 1011} 0 = (f(\overline{e})|\overline{f(\overline{e})}) = \frac14 (\mu_1^2-\mu_2^2 - i \mu_1 {\mu_2} (y_1 | y_2) -i \mu_2 {\mu_1} (y_2|y_1)).
\end{equation} Combining \eqref{eq 2 1011}, and \eqref{eq 3 1011} we get $\mathbb{R}\ni(y_1|y_2)=0$ and $\mu_1^2 = \mu_2^2,$ and hence $\mu_2 =\pm \mu_1$.\smallskip

In the second step we shall prove the existence of a complex number $\mu_0$ with $|\mu_0|=1$ satisfying $$f( x_1 ) \in \mu_0 S(Y_1), \hbox{ for every } x_1\in S(X_1).$$  For this purpose, pick a norm one element $x_0$ in $X_1$. As before, there exist $\mu_0\in \mathbb{C}$ and $y_0\in S(Y_1)$ such that $f(x_0) = \mu_0 y_0$. Let $x_1$ be any element in $S(X_1)$. We can find a third element $x_2\in S(X_1)$ satisfying $(x_1|x_2) =(x_0|x_2)=0$. Applying the first step to the pairs $(x_0,x_2)$ and $(x_1,x_2)$ we get $$f(x_2) =\mu_0 z_2, \hbox{ and } f(x_1 ) = \mu_0 z_1,$$ for suitable $z_2, z_1\in S(Y_1)$. This finishes the proof of the second step.\smallskip

The sets $X_1$ and $\mu_0 Y_1$ are real linear closed subspaces of $X$ and $Y$, respectively. Furthermore, the spin norms on $X_1$ and on $\mu_0 Y_1$ are precisely the Hilbertian norms associated to the inner products $(.|.)$, in other words, $(X_1, \|.\|) = (X_1, \|.\|_2)$ and $(\mu_0 Y_1, \|.\|) = (\mu_0 Y_1, \|.\|_2)$. We have proved in the second step that $$f|_{S(X_1)} : S(X_1) \to S(\mu_0 Y_1)$$ is a surjective isometry between the unit spheres of two Hilbert spaces. We deduce from \cite{Ding2002} (see also Corollary 3.15 in \cite{PeTan16}) the existence of a surjective real linear isometry $F : X_1\to \mu_0 Y_1$ satisfying $F(x_1 ) = f(x_1)$ for every $x_1\in S(X_1)$. We define a (complex) linear operator $T : X\to Y$ given by $T(x_1 + i z_1) := F(x_1) + i F(z_1)$ ($x_1+i z_1\in X = X_1 \oplus i X_1$).\smallskip

Every minimal tripotent in $X$ is of the form $e = \frac{x_1+i z_1}{2}$, where $x_1,z_1\in S(X_1)$ and $(x_1|z_1)=0$. The elements $u = x_1$ and $w= i z_1$ are complete tripotents in $X$. For each $t\in [0,1]$ we have  $$ \| t u + (1-t) w \|_2^2 = t^2 + (1-t)^2, $$ $$ |(t u + (1-t) w \ |\ t \overline{u} + (1-t) \overline{w} )|^2 =(t^2 - (1-t)^2)^2,$$ and hence
$$\| t u + (1-t) w \|^2 = t^2 + (1-t)^2 + \sqrt{(t^2 + (1-t)^2)^2 - (t^2 - (1-t)^2)^2} =1.$$ This shows that $[u,w] =\{ t u + (1-t) w : t\in [0,1]\}$ is a convex subset in $S(X)$, and by Zorn's lemma, it is contained in some maximal proper norm closed face $M$ of $\mathcal{B}_X$. We conclude, by Proposition \ref{p surjective isometries between the spheres preserve finite rank tripotents}$(b)$, that $$ f(e) = f\left(\frac12 u +\frac12 w\right) = \frac12 f(u) + \frac12 f(w)= \frac12 f(x_1) + \frac12 f(i z_1) $$ $$= \frac12 f(x_1) +i \frac12 f( z_1) =\hbox{(by definition of $T$)}= \frac12 T(u) + \frac12 T(w) =T(e).$$ Therefore, $f(e) = T(e)$ for every minimal tripotent $e\in X$. An application of Proposition 3.9 in \cite{PeTan16} proves that $f(x) = T(x)$ for every $x\in S(X)$, and hence $T$ is a surjective linear isometry.\smallskip

We finally observe that if we assume that $f(i e) = -i f(e)$ for every minimal tripotent $e$ in $X$, then the above arguments show the existence of a conjugate linear isometry $T: X\to Y$ satisfying  $f(x) = T(x)$ for every $x\in S(X)$.
\end{proof}

\textbf{Acknowledgements} Authors partially supported by the Spanish Ministry of Economy and Competitiveness and European Regional Development Fund project no. MTM2014-58984-P and Junta de Andaluc\'{\i}a grant FQM375.


\begin{thebibliography}{0}



\bibitem{BarTi86} T. Barton abd R. M. Timoney, Weak$^*$-continuity of Jordan triple products
 and applications, \emph{Math. Scand.} \textbf{59} (1986), 177--191.

\bibitem{BeLoPeRo} J. Becerra Guerrero, G. L{\'o}pez P{\'e}rez, A. M. Peralta, A. Rodr{\'\i}guez-Palacios,
A. Relatively weakly open sets in closed balls of Banach spaces,
and real ${\rm JB}^*$-triples of finite rank, \emph{Math. Ann.}
\textbf{330} (2004),  no. 1, 45--58.

\bibitem{BuChu} L.J. Bunce and C.-H. Chu,  Compact  operations, multipliers  and
Radon-Nikodym property in $JB^*$-triples, {\it Pacific J. Math.} \textbf{153} (1992), 249--265.

\bibitem{BuFerMarPe} L.J. Bunce, F.J. Fern\'andez-Polo, J. Mart{\'i}nez Moreno, A.M. Peralta,
A Sait\^o-Tomita-Lusin theorem for JB$^*$-triples and applications,
Quart. J. Math. Oxford, {\bf 57} (2006), 37-48.

\bibitem{BurFerGarMarPe} M. Burgos, F.J. Fern{\'a}ndez-Polo, J. Garc{\'e}s, J.
Mart{\'\i}nez, A.M. Peralta, Orthogonality preservers in
C$^*$-algebras, JB$^*$-algebras and JB$^*$-triples, \emph{J. Math. Anal.
Appl.} \textbf{348}, 220-233 (2008).


\bibitem{Chu2012} Ch.-H. {Chu}. {\em {Jordan structures in geometry and analysis.}}, Cambridge, Cambridge University Press, 2012.


\bibitem{Da} T. Dang, Real isometries between JB$^*$-triples, \emph{Proc. Amer. Math. Soc.} \textbf{114}, 971-980 (1992).

\bibitem{DanFri87} T. Dang, Y. Friedman, Classification of JBW$^*$-triple factors and applications, \emph{Math. Scand.} \textbf{61}, no. 2, 292-330 (1987).


\bibitem{Di86} S. Dineen, The second dual of a JB$^*$-triple system, In: Complex analysis,
functional analysis and approximation theory (ed. by J. M\'ugica),
67-69, (North-Holland Math. Stud. 125), North-Holland,
Amsterdam-New York, (1986).

\bibitem{Ding2002} G. Ding, The 1-Lipschitz mapping between the unit spheres of two Hilbert spaces can be extended to a real linear isometry of the whole space, \emph{Sci. China Ser. A} \textbf{45} (2002), no. 4, 479-483.

\bibitem{Di:p}
G. G. Ding, The isometric extension problem in the spheres of $l^p (\Gamma)$ $(p>1)$ type spaces, \emph{Sci.\ China Ser.\ A} \textbf{46} (2003), 333--338.


\bibitem{Di:8}
G. G. Ding, The representation theorem of onto isometric mappings between two unit spheres of $l^\infty$-type spaces and the application on isometric extension problem, \emph{Sci.\ China Ser.\ A} \textbf{47} (2004), 722--729.

\bibitem{Di:1}
G. G. Ding, The representation theorem of onto isometric mappings between two unit spheres of $l^1 (\Gamma)$ type spaces and the application to the isometric extension problem, \emph{Acta.\ Math.\ Sin.\ (Engl.\ Ser.)} \textbf{20} (2004), 1089--1094.

\bibitem{Ding2009} G. Ding, On isometric extension problem between two unit spheres, \emph{Sci. China Ser. A} \textbf{52} (2009) 2069-2083.


\bibitem{EdFerHosPe2010} C.M. Edwards, F.J. Fern{\'a}ndez-Polo, C.S. Hoskin, A.M. Peralta, On the facial structure of the unit ball in a JB$^*$-triple, \emph{J. Reine Angew. Math.} \textbf{641} (2010) 123-144.


\bibitem{EdRutt88} C.M. Edwards, G.T. R\"{u}ttimann, On the facial structure of the unit balls in a JBW$^*$-triple and its predual, \emph{J. Lond. Math. Soc.} \textbf{38}, 317-322 (1988).

\bibitem{EdRu96} C.M. Edwards, G.T. R\"uttimann, Compact tripotents in bi-dual JB$^*$-triples, Math. Proc.
Camb. Phil. Soc. \textbf{120}, 155-173 (1996).

\bibitem{FerMarPe} F.J. Fern{\'a}ndez-Polo, J. Mart{\' i}nez, A.M. Peralta, Surjective isometries
between real JB$^*$-triples, \emph{Math. Proc. Cambridge Phil. Soc.}, \textbf{137} 709-723 (2004).







\bibitem{FriRu85}
Y.~{Friedman} and B.~{Russo}.
\newblock {Structure of the predual of a $JBW\sp*$-triple.}
\newblock {\em {J. Reine Angew. Math.}}, 356:67--89, 1985.



\bibitem{Horn87} G. Horn, Characterization of the predual and ideal structure of a JBW$^*$-triple, \emph{Math. Scand.} \textbf{61}, no. 1, 117-133 (1987).


\bibitem{JamPeSiTah2014Ceby} F.B. Jamjoom, A.M. Peralta, A.A. Siddiqui, H.M. Tahlawi, \v{C}eby\v{s}\"{e}v subspaces of JBW$^{*}$-triples, \emph{J. Inequal. Appl.} (2015) 2015:288. DOI 10.1186/s13660-015-0813-2.

\bibitem{KadMar2012} V. Kadets and M. Mart{\'i}n, Extension of isometries between unit spheres of infite-dimensional polyhedral Banach spaces, \emph{J. Math. Anal. Appl.}, \textbf{396} (2012), 441-447.


\bibitem{Ka81} W. Kaup, \"{U}ber die Klassifikation der symmetrischen hermiteschen Mannigfaltigkeiten unendlicher
Dimension I, \emph{Math. Ann.} \textbf{257}, 463-486 (1981).

\bibitem{Ka83} W. Kaup, A Riemann Mapping Theorem for bounded symmentric domains in complex Banach spaces, \emph{Math. Z.} \textbf{183}, 503-529 (1983).

\bibitem{Ka97} W. Kaup, On real Cartan factors, \emph{Manuscripta Math.} \textbf{92}, 191-222 (1997).

\bibitem{KaUp} W. Kaup and H. Upmeier, Jordan algebras and
symmetric Siegel domains in Banach spaces, \emph{Math. Z.}
\textbf{157}, 179-200 (1977).


\bibitem{Mank1972} P. Mankiewicz, On extension of isometries in normed linear spaces, \emph{Bull. Acad. Pol. Sci., S\'{e}r. Sci. Math. Astron. Phys.} \textbf{20} (1972) 367-371.


\bibitem{Neher87} E. Neher, \emph{Jordan triple systems by the grid approach}, Lecture Notes in Mathematics, \textbf{1280}. Springer-Verlag, Berlin, 1987.


\bibitem{PeTan16} A.M. Peralta, R. Tanaka, A solution to Tingley's problem for isometries between the unit spheres of compact C$^*$-algebras and JB$^*$-triples, preprint 2016. arXiv:1608.06327v1.


\bibitem{Ta:8} D. Tan, Extension of isometries on unit sphere of $L^\infty$, \emph{Taiwanese J. Math.} \textbf{15} (2011), 819--827.

\bibitem{Ta:1} D. Tan, On extension of isometries on the unit spheres of $L^p$-spaces for $0<p \leq 1$, \emph{Nonlinear Anal.} \textbf{74} (2011), 6981-6987.

\bibitem{Ta:p} D. Tan, Extension of isometries on the unit sphere of $L^p$-spaces, \emph{Acta.\ Math.\ Sin.\ (Engl.\ Ser.)} \textbf{28} (2012), 1197--1208.





\bibitem{Tan2016-2} R. Tanaka, Spherical isometries of finite dimensional $C^*$-algebras, \emph{J. Math. Anal. Appl.} \textbf{445} (2017), no. 1, 337-341.

\bibitem{Tan2016preprint} R. Tanaka, Tingley's problem on finite von Neumann algebras, preprint 2016.

\bibitem{Ting1987} D. Tingley, Isometries of the unit sphere, \emph{Geom. Dedicata} \textbf{22} (1987) 371-378.

\bibitem{Wang} R.S. Wang, Isometries between the unit spheres of $C_0(\Omega)$ type spaces, \emph{Acta Math. Sci.} (English Ed.) \textbf{14} (1994), no. 1, 82-89.

\bibitem{YangZhao2014} X. Yang, X. Zhao, On the extension problems of isometric and nonexpansive mappings. In: \emph{Mathematics without boundaries}. Edited by Themistocles M. Rassias and Panos M. Pardalos. 725-748, Springer, New York, 2014.



\end{thebibliography}
\end{document}